\documentclass[10pt, reqno,amsmath,amsthm,amssymb,amscd]{amsart}
\usepackage{amsmath}
\usepackage{amssymb}
\usepackage{verbatim}
\usepackage{dsfont}
\usepackage{enumerate}
\usepackage[all]{xy}
\usepackage[mathscr]{eucal}

\usepackage{amsthm}
\usepackage{stmaryrd}
\usepackage{amscd}
\usepackage{amsfonts}
\usepackage{latexsym}

\usepackage{amscd}

\usepackage{graphicx}

\usepackage{amsmath}
\usepackage{mathrsfs,amssymb, amscd,amsmath,amsthm}
\usepackage[enableskew,vcentermath]{youngtab}
\usepackage{multicol}\multicolsep=0pt
\usepackage{tikz}

 \providecommand{\og}{``}
\providecommand{\fg}{''}

\usepackage{amssymb}
\usepackage{mathrsfs,amssymb}
\usepackage{multicol}\multicolsep=0pt
\usepackage{pstricks,pst-node}

\usepackage[enableskew,vcentermath]{youngtab}
\usepackage[sort]{cite}
\usepackage{xcolor,graphicx}

\def\crulefill{\leavevmode\leaders\hrule height 1pt\hfill\kern 0pt}
\long\def\QUERY#1{%
\leavevmode\newline%
\noindent$\star\star\star$\thinspace\textsf{Comment/Query}\crulefill\newline%
   \space #1\newline\hbox to 120mm{\crulefill}$\star\star\star$\newline}
\newtheorem{Theorem}{Theorem}[section]%[chapter] theorem number will %continue
\newtheorem{Lemma}[Theorem]{Lemma}
\newtheorem{Cor}[Theorem]{Corollary}
\newtheorem{Prop}[Theorem]{Proposition}

 \theoremstyle{definition}

\newtheorem{Defn}[Theorem]{Definition}

\newtheorem{rem}[Theorem]{Remark}
\newtheorem{rems}[Theorem]{Remarks}

\numberwithin{equation}{section}
\theoremstyle{definition}
\newtheorem{THEOREM}{Theorem}

\makeatletter
\def\enumerate{\begingroup\ifnum\@enumdepth>3\@toodeep\else
      \advance\@enumdepth\@ne
      \edef\@enumctr{enum\romannumeral\the\@enumdepth}%
      \topsep\z@\parskip\z@
      \list{\csname label\@enumctr\endcsname}
        {\@nmbrlisttrue\let\@listctr\@enumctr
         \parsep\z@\itemsep\z@\topsep\z@
         \setcounter{\@enumctr}{0}
         \def\makelabel##1{\hss\llap{\rm ##1}}
       }\fi}

\makeatother

\let\bar=\overline
\let\epsilon=\varepsilon
\def\({\big(}
\def\){\big)}

\def\0{\underline{0}}

\def\Std{\mathscr{T}^{std}}

\def\s{\mathfrak s}

\def\t{\mathfrak t}
\def\u{\mathfrak u}

\def\Hom{\text{Hom}}

{\catcode`\|=\active
  \gdef\set#1{\mathinner{\lbrace\,{\mathcode`\|"8000%
                                   \let|\midvert #1}\,\rbrace}}
  \gdef\seT#1{\mathinner{\Big\lbrace\,{\mathcode`\|"8000%
                                   \let|\midverT #1}\,\Big\rbrace}}
}
\def\midvert{\egroup\mid\bgroup}
\def\midverT{\egroup\,\Big|\,\bgroup}

\def\Set[#1]#2|#3|{\Big\{\ #2\ \Big| \
           \vcenter{\hsize #1mm\centering #3}\Big\}}

\def\qed{\hfill\mbox{$\Box$}}

\def\Hom{{\rm Hom}}

\def\Set{{\rm Set}}

\def\Std{\mathscr{T}^{std}}%
\def\s{\mathfrak s}%
\def\t{\mathfrak t}%
\def\u{\mathfrak u}%
\def\Hom{\text{Hom}}%
\def\textsf#1{{\textit{#1}}}%

\definecolor{white}{HTML}{FFFFFF}
\definecolor{darkblue}{HTML}{111199}
\definecolor{darkgreen}{HTML}{336633}
\definecolor{darkred}{HTML}{993333}

\definecolor{darkpurple}{HTML}{995599}

\begin{document}
\title[{\tiny  The Jucys-Murphy basis and semisimplicty criteria for the   $q$-Brauer algebra}]{ The Jucys-Murphy basis and semisimplicty criteria for the  $q$-Brauer algebra}
\author{ Hebing Rui, Mei Si  and  Linliang Song} \address{H.R.  School of Mathematical Science, Tongji University,  Shanghai, 200092, China}\email{hbrui@tongji.edu.cn}
\address {M.S. School of Mathematics, Shanghai Jiaotong University, Shanghai, 200240, China}\email{simei@sjtu.edu.cn}
\address{L.S.  School of Mathematical Science, Tongji University,  Shanghai, 200092, China}\email{llsong@tongji.edu.cn}
\thanks{ H. Rui is supported  partially by NSFC (grant No.  11971351). M. Si and  L. Song are supported  partially by NSFC (grant No.  12071346) and  China Scholarship Council. }
\date{\today}
\sloppy
\maketitle

\begin{abstract} We construct the Jucys-Murphy  elements and the Jucys-Murphy basis for the $q$-Brauer algebra
in the sense of \cite[Definition~2.4]{MA1}. We also give a necessary and sufficient condition for  the $q$-Brauer algebra being (split)  semisimple  over an arbitrary field.
   \end{abstract}

\section{Introduction}
In his remarkable work~\cite{Schu0}, Schur classified the polynomial representations of the complex general
linear group
 by decomposing tensor products of its  natural representation. The  symmetric group
 appears  naturally in Schur's work and plays an important role.  Later on Brauer~\cite{B} studied the related problem for  the complex orthogonal group
 and the complex symplectic group
 and found a new family of complex associative algebras called the Brauer algebras or Brauer centralizer algebras. These algebras play   roles similar to those  of symmetric groups.

There are many ways to generalize results in ~\cite{Schu0, B}. See \cite{Jim, Wen} where the complex classical groups are replaced by
quantum enveloping algebras of classic types.  Hecke algebras and Birman-Murakami-Wenzl algebras~\cite{Bir, Mur} come into the picture.
%Both of them are closely connected with knot invariants, subfactors  and tensor categories\cite{Wen, TW}.
In ~\cite{Wen1}, Wenzl considered  tensor products of the natural representation of   certain  coideal algebras \cite{Let} and introduced a new  associative algebra, called the  $q$-Brauer algebra    $\mathcal B_n(q,z)$  for any positive integer $n\ge 2$.
See also \cite{Cui} for the explicit Schur-Weyl duality between 
the $\imath$-quantum groups of types AI or AII and the  $q$-Brauer algebras.  
    The  $q$-Brauer algebra is isomorphic to
   an endomorphism algebra of the $q$-Brauer  category~ \cite{SDT} and is isomorphic to
    a  quotient of
the  associative  algebra introduced by Molev~\cite{Mol}.

The  $q$-Brauer algebra has been used to construct  new subfactors of type II$_1$ factors~\cite{Wen2}. It has some application in  module categories of fusion categories of type $A$ corresponding to certain symmetric spaces\cite{Wen10}.
 In \cite[Theorem~3.8]{Wen1}, Wenzl proved  that $\mathcal B_n(q, z)$ is free over a commutative ring with rank $(2n-1)!!$.
When the ground ring is a field,
Wenzl also proved that
   $\mathcal B_n(q,z)$ is semisimple  if $z^2\neq q^{2k}$ for $k\in \mathbb Z$, $|k|\le n$ and $q^2$ is not an $\ell$-th root of unity, $\ell\le n$~\cite[Theorem~5.3]{Wen1}. Under the assumption $z^2\neq 1$,
   Dung Tien Nguyen proved that  $\mathcal B_n(q,z)$ is a cellular algebra in the sense of \cite[Definition~1.1]{GL} and  classified simple $\mathcal B_n(q,z)$-modules over an arbitrary field~\cite[Theorem~3.2, Theorem~4.1]{N21}.
He also  gave an explicit condition for $\mathcal B_3(q, z)$ being  semisimple over an arbitrary field~\cite[Proposition~5.1]{N21}.  Dung Tien Nguyen's conditions on
 parameters \cite[Example~5.7]{N21} are different  from  those  for the corresponding  Birman-Murakami-Weznl algebra in \cite[Theorem~5.9]{RS1}.
 Therefore the   $q$-Brauer algebra  is another $q$-deformation of the Brauer algebra.
 %different from the Birman-Murakami-Wenzl algebra.

 We expect to generalize our previous results in ~\cite{RS1, RS2} to the $q$-Brauer algebra. More explicitly, we want to construct Jucys-Murphy basis and orthogonal basis of the $q$-Brauer algebra so as to compute Gram determinants, determine explicitly the semisimplicity criteria,  classify blocks and compute decomposition matrices (when $q$ is not a root of unity) etc. The first difficult problem that we faced is to find out  a family of commutative elements in $\mathcal B_n(q,z)$ called the Jucys-Murphy elements or  JM elements in the sense of \cite[Definition~2.4]{MA1}. By studying classical branching rule for the $q$-Brauer algebra in details, we are able to  define the Jucys-Murphy basis of the  $q$-Brauer algebra.
 Since the summation of Jucys-Murphy elements is not a central element,
  we have to make extra efforts
to prove that Jucys-Murphy elements act on Jucys-Murphy basis as  upper-triangular matrices. As a by-product one can use arguments in \cite{MA1} to define an orthogonal basis of $\mathcal B_n(q,z)$ in generic case. In this paper,  we do not give details  since we have not found out an efficient way to  compute
Gram determinants associated to  cell modules of $\mathcal B_n(q, z)$, and
  we could not use the  method in \cite{RS1} to give  a  semisimplicity criteria  for
$\mathcal B_n(q, z)$  over an arbitrary field. We describe explicitly the radicals of certain cell modules in section~six and finally are able to
  prove the following result:

\begin{THEOREM}\label{main3} For any $n\ge 2$, let $\mathcal B_n( q, z)$ be defined over a field $F$ which  contains invertible elements $z$, $z-z^{-1}$, $q$ and $q-q^{-1}$. Let $e$ be the quantum characteristic of $q^2$. Then
 $\mathcal B_n(q, z)$ is (split) semisimple if and only if
$e>n$ and $z^2\neq q^{2a}$ where
\begin{equation}\label{paracon} a\in \{i\in \mathbb Z\mid 4-2n\le i\le n-2\}\setminus \{i\in \mathbb Z\mid
4-2n<i\le 3-n, 2\nmid i\}.\end{equation}
\end{THEOREM}

Here is a brief outline of the content of this paper.  In the second section,  we recall
some basic results on the $q$-Brauer algebra. In section 3, we study the classical branching rule. In section 4,  we  construct the Jucys-Murphy basis and the Jucys-Murphy  elements of the $q$-Brauer algebra. Restriction and induction functors are given in section 5. We describe
explicitly the radical of certain cell modules in section 6. This enables us
to prove Theorem~\ref{main3} in section 7.

\section{The $q$-Brauer algebra}
In this section, we recall the definitions of the Hecke algebra and the $q$-Brauer algebra and give some basic results on them.

\subsection{The Hecke algebra associated to the symmetric group } The symmetric group $\mathfrak S_n$  is  the Coxeter group with $\{s_i\mid 1\le i\le n-1\}$ as its distinguished generators subject to the relations:
\begin{equation} \label{sysg} s_i^2=1,\ \  s_is_j=s_js_i,\ \ s_ks_{k+1}s_k=s_{k+1}s_ks_{k+1}\end{equation}
where  $1<|i-j|$ and $1\le k\le n-2$.
In this paper we assume $\mathcal Z=\mathbb Z[q, q^{-1}]$, the ring of Laurent polynomials  in indeterminate $q$ with coefficients in $\mathbb Z$. The Hecke algebra $\mathcal H_n$ associated to $\mathfrak S_n$
  is the unital associative $\mathcal Z$-algebra generated by $\{ T_i\mid  1\leq i\leq n-1\}$ subject to  the relations
\begin{equation} \label{kkk1} (T_i-q)(T_i+q^{-1})=0,\ \  T_iT_j=T_jT_i,\ \ T_kT_{k+1}T_k=T_{k+1}T_kT_{k+1}\end{equation}
where   $1<|i-j|$ and $1\le k\le n-2$.

  Suppose $w\in \mathfrak S_n$.  Then  $s_{i_1}s_{i_2}\cdots s_{i_k}$ is a reduced expression of  $w$  if $w=s_{i_1}s_{i_2}\cdots s_{i_k}$ and $k$ is minimal. In
   this case, $k$ is the length of $w$ and is denoted by $\ell(w)$.
   Let $T_w=T_{i_1}T_{i_2}\cdots T_{i_k}$ if $s_{i_1}s_{i_2}\cdots s_{i_k}$ is a reduced expression of $w$. It is known that $T_w$ is independent of any reduced expression of $w$ and $\{T_w\mid w\in \mathfrak S_n\}$ is a $\mathcal Z$-basis of $\mathcal H_n$. Let $\tau : \mathcal H_n\rightarrow \mathcal Z$ such that $\tau (\sum_w a_w T_w)=a_1$ where $1$ is the identity element of $\mathfrak S_n$.
Then $\tau$ is a trace function on $\mathcal H_n$ such that   \begin{equation}\label{tau} \tau(T_xT_y)=\delta_{x, y^{-1}} \end{equation} for all $x, y\in \mathfrak S_n $~\cite[Proposition~1.16]{Ma}.
%We will need \eqref{tau} in  section~6.

\subsection{The $q$-Brauer algebra}
Let $\Bbbk=\mathbb Z[q,q^{-1}, z, z^{-1}, (q-q^{-1})^{-1}, (z-z^{-1})^{-1}]
$, where  $z$ is another  indeterminate. In this paper we always assume \begin{equation}\label{delta}\delta=\frac{z-z^{-1}}{q-q^{-1}}.\end{equation}
Then $\delta$ is invertible  in $\Bbbk$.
The original $q$-Brauer algebra is available over $\mathbb Z[q,q^{-1}, z, z^{-1}, (q-q^{-1})^{-1}]$. In this case,
 $\delta$ may  not be invertible.  In \cite{N18}, Dung Tien Nguyen proved that $q$-Brauer algebra is a cellular algebra in the sense of \cite[Definition~1.1]{GL} under the assumption that $\delta$ is invertible.
  Since our results depend on the cellular structure of the $q$-Brauer algebra, we need this additional assumption.
\begin{Defn}\label{qqq} \cite[Definition~3.1]{Wen1} \cite[Definition~3.1]{N21} Suppose $n\ge 2$. The $q$-Brauer algebra $\mathcal B_n(q, z)$ or just $\mathcal B_n$  is a unital associative $\Bbbk$-algebra generated by $T_1, T_2, \cdots, T_{n-1}$ and $E_1$ subject to the relations in \eqref{kkk1} together with
\begin{multicols}{2}
\item[(1)] $E_1^2 =\delta E_1$,
\item[(2)] $T_1 E_1=qE_1=E_1T_1$,
\item[(3)] $E_1T_2E_1 =z E_1$,
\item[(4)] $T_i E_1=E_1T_i$ if $i>2$,
\item[(5)] $T_2T_3T_1^{-1}T_2^{-1} E_{(2)}=E_{(2)}=E_{(2)} T_2T_3T_1^{-1}T_2^{-1}$,
\end{multicols}
where $E_{(2)}=E_1 T_2T_3T_1^{-1}T_2^{-1} E_1$.
\end{Defn}

\begin{rems}\label{limit} Suppose $z=q^a$ for some $a\in \mathbb Z\setminus\{0\}$. It follows from
\cite[Remark~3.1]{Wen1} that the classical limit of
$\mathcal B_n(q, q^a)$ is the Brauer algebra $B_n(a)$ over $\mathbb Z$. In this case, $\lim_{q\rightarrow 1} \delta=a$ and $T_i$ becomes the simple reflection $s_i$ and $E_1$ can be identified with the corresponding element in $B_n(a)$. \end{rems}
We write $\mathcal B_0=\mathcal B_1=\Bbbk$.
Thanks to  \cite[Theorem~3.8]{Wen1},  																																																																																																																																																																	 $\mathcal B_n$ is free over $\Bbbk$ with rank $(2n-1)!!$ and   the subalgebra of  $\mathcal B_n$ generated by $T_1, T_2, \ldots, T_{n-1}$ is isomorphic to the Hecke algebra $\mathcal H_n$. For this reason, one can  use symbols $T_i$'s in both $\mathcal H_n$ and   $\mathcal B_n$. Moreover, \begin{equation}\label{quot1}  \mathcal H_n\cong \mathcal B_{n}/\langle E_1\rangle\end{equation} where $\langle E_1\rangle$ is the two-sided ideal of $\mathcal B_n$ generated by $E_1$.
Let $  F$ be a field containing units  $z,q,q-q^{-1}, z-z^{-1}$. Throughout, $$\mathcal B_{n,  F}:= \mathcal B_n\otimes _{\Bbbk} F ,$$
 where $F$ can be considered as the left $\Bbbk$-module on which $z, q\in \Bbbk$ act  via the corresponding $z$ and $q$ in $F$.
 If there is no confusion, we also denote $\mathcal B_{n,  F}$ by $\mathcal B_{n}$.

%The following result was given in \cite[Proposition~3.14]{N18}.
\begin{Lemma}\label{anti}\cite[Proposition~3.14]{N18} There is a $\Bbbk$-linear anti-involution $\sigma: \mathcal B_n\rightarrow\mathcal B_n$ such that $\sigma(x)=x$ for any $x\in \{E_1, T_1, T_2, \ldots, T_{n-1}\}$.    \end{Lemma}

\begin{Defn}\label{k2222}  Suppose $0\le k\le \lfloor \frac n 2\rfloor -1$ and $1\le \ell \le n-1$. Define \begin{itemize}\item [(1)] $E^0=1$  and  $E^{k+1}
=E_1 T_{2, 2k+2} T_{2k+1, 1}^{-1} E^k$,
where $T_{i,j}=T_{s_{i,j}}$ and $$s_{i,j}=\begin{cases} s_is_{i+1, j} &\text{if $i<j$,}\\  1 &\text{if $i=j$,}\\ s_{i, j+1}s_j &\text{if $i>j$.}\\
\end{cases}$$
\item [(2)] $E_\ell=T_{\ell,1} T_{2,\ell+1}^{-1} E_1 T_{2, \ell+1} T_{\ell, 1}^{-1}$.
\end{itemize}
\end{Defn}
Definition~\ref{k2222}(1)  was given in   \cite[(3.2)]{Wen1} where  $E^k$ was denoted by $E_{(k)}$.

\begin{Lemma}\cite[(3.2), Lemma~3.3]{Wen1}, \cite[Lemma~3.4, Proposition~3.14]{N18}\label{k2}  Suppose  $1\le k\le \lfloor \frac n 2\rfloor $ and  $1\le \ell\le \lfloor \frac n 2\rfloor -1$.
\begin{multicols}{2}
\item[(1)] $E^{\ell+1}=E^{\ell} T_{1, 2\ell+1}^{-1}  T_{2\ell+2, 2} E_1$,
\item[(2)] $\sigma(E^{k})=E^{k}$,
\item[(3)] $E_1T_2^{-1}T_1^{-1} T_3T_2E_1=E^2=E_1T_2T_1T_3^{-1}T_2^{-1} E_1$,
\item[(4)] $E^{k} T_{2l+1, 1} =E^{k} T_{2, 2l+2}$ if $l<k$,
\item[(5)] $E^{k} T_{1, 2l+1}^{-1} =E^{k} T_{2l+2, 2}^{-1}$ if $l<k$,
\item [(6)] $E^k T_{2k-1}=qE^{k}$.
\end{multicols}
\end{Lemma}

\begin{Lemma}\label{k3}For all admissible $i$ and $j$, we have:
\begin{multicols}{2}
\item[(1)] $E_{i+1}=T_i T_{i+1}^{-1} E_{i} T_{i+1} T_i^{-1}$,
\item[(2)] $T_i E_{j}=E_{j}T_i$  if $|i-j|\geq2$,
\item[(3)] $E_{i}E_j=E_jE_i $  if $|i-j|\geq2$,
\item[(4)] $E^{j}  =E_{1}E_3\cdots E_{2j-1}$,
\item[(5)] $E_{i} T_{i+1}^{\pm 1} E_i  =z^\pm E_{i} $,
\item[(6)] $E_i^2=\delta E_i$,
\item [(7)]  $ E_1T_{2i,2}^{-1}T_{1,2i-1}E^{i-1}=E^{i}$~\cite[Lemma~3.3]{Cui},
\item [(8)] $E_{2i-1}T_{ 2i-1}^{-1}T_{ 2i}^{-1} E^{i}=z^{-1}q^{-1}E^i$.
\end{multicols}
\end{Lemma}
\begin{proof} Obviously (1) follows immediately from  Definition~\ref{k2222}(2).
If $i\ge j+2$, then (2) follows from    \eqref{kkk1} and Definition~\ref{qqq}(4). If $i\le j-2$, then
$T_i E_j=T_{j,1} T_{2,j+1}^{-1} T_{i+2} E_1 T_{2, j+1} T_{j, 1}^{-1}=E_jT_i$ and (2) follows.
 It is not difficult to verify
\begin{equation} \label{e2} E_1E_3=E_{(2)}=E_3E_1.\end{equation} If $j>3$, then
$E_1E_j=E_1 T_{j-1} T_j^{-1} E_{j-1} T_{j}T_{j-1}^{-1}=E_jE_1$, where the last equality  follows from the induction assumption on $j-1$ and Definition~\ref{qqq}(4).
 Suppose $i>1$. We can assume $j\ge i+2$ without losing of  any generality when we prove (3).  By (2) and induction assumption on $i-1$,
$E_iE_j=T_{i-1}T_i^{-1}  E_{i-1} T_i T_{i-1}^{-1} E_j=E_j E_i$. In any case, we have verified  (3).
 Thanks to \eqref{e2},  $E^2=E_1E_3$. In general, using (3), induction assumption on $j-1$, Lemma~\ref{k2}(5) and $\sigma$ in Lemma~\ref{anti} yields
$$
\begin{aligned} E_1E_3\cdots E_{2j-1}
&=T_{2j-1,1} T_{2,2j}^{-1} E_1 T_{2, 2j}T_{2j-1,1}^{-1} E^{j-1}\\
&=T_{2j-1,1}
T_{2,2j}^{-1} E^{j}\\
&=E^j,\end{aligned}$$ proving (4).
 $E_1 T_2E_1=zE_1$  is  given   in Definition~\ref{qqq}(3).  The required formula for $E_1 T_2^{-1} E_1$ follows from the quadratic relation in \eqref{kkk1} and Definition~\ref{qqq}(3). In general, by  induction assumption on $i-1$,   we have
$$ E_i T_{i+1}^{\pm 1} E_i  =T_{i-1}T_{i}^{-1}E_{i-1} T_{i+1}^{-1} T_{i}^{\pm 1} T_{i+1} E_{i-1} T_iT_{i-1}^{-1}=z^{\pm 1} E_i,$$ proving (5).  Obviously,   (6) follows from Definition~\ref{qqq}(1) and
(7) was given in \cite[Lemma~3.3]{Cui}. Finally,  (8) follows from the following computation:
$$\begin{aligned}
 E_{2i-1}T_{ 2i-1}^{-1}T_{ 2i}^{-1} E^{i}=&T_{2i-1,1}T_{2,2i}^{-1} E_1T_{2i+1,1}^{-1}T_{1, 2i-1}E^i\\
=&
 q^{-1} T_{2i-1,1}T_{2,2i}^{-1} E_1T_{2i+1,2}^{-1}T_{1,2i-1}E^i\\
=& q^{-1} T_{2i-1,1}T_{2,2i}^{-1} E_1T_{2i,2}^{-1}T_{1,2i-1}E^{i-1} T_{2i}^{-1}E_{2i-1}\\
 = &q^{-1} T_{2i-1,1}T_{2,2i}^{-1}E^i T_{2i}^{-1}E_{2i-1}\\ = & z^{-1}q^{-1} T_{2i-1,1}T_{2,2i}^{-1}E^i\\ = &z^{-1}q^{-1}E^i
\end{aligned} $$
where the forth (resp., fifth, resp., sixth) equation follows from (7) (resp., (5), resp., Lemma~\ref{k2}(2)(5)).
\end{proof}

%Note that   $E_iT_i\neq qE_i$ and $T_i E_i\neq q E_i$ if $i>1$.

\begin{Prop}\label{ffunc1} Suppose $n\ge 2$ and  $ \tilde E_1=qz^{-1} T_1^{-1} T_2E_1$. Then  $\tilde E_1^2=\tilde E_1$. Moreover,
\begin{itemize}
\item[(1)] There is an algebra homomorphism
$\phi: \mathcal B_{n-2} \rightarrow \tilde E_1 \mathcal B_n \tilde E_1$ sending $1, E_1, T_i$ to $\tilde E_1$,  $\tilde E_1E_3$ and $\tilde E_1 T_{i+2}$, $ 1\le i\le n-3$, respectively.
\item[(2)] $\phi (E^i)=qz^{-1} T_1^{-1} T_2E^{i+1}$ for any $0\le i\le\lfloor \frac{ n}{2}\rfloor-1$.
\end{itemize} \end{Prop}
 \begin{proof} It is easy to verify that $\tilde E_1^2=\tilde E_1$,
  $\tilde E_1 E_3 \tilde E_1=\tilde E_1 E_3$ and $\tilde E_1T_i\tilde E_1=\tilde E_1T_i$ if $3\le i\le n-1$. So
  $\phi(E_1),  \phi(T_1), \ldots, \phi(T_{n-3})$ satisfy \eqref{kkk1} and Definition~\ref{qqq}(1)-(4). We remark that the quadratic relation in
  \eqref{kkk1} becomes $(\tilde E_1 T_i-q\tilde E_1) (\tilde E_1 T_i+q^{-1} \tilde E_1)=0$, $3\le i\le n-1$. We have
 $$\begin{aligned} \phi(E^2) & = \tilde  E_1 E_3 T_4T_5T_3^{-1} T_4^{-1}\tilde E_1 E_3
\\
 & = qz^{-1}T_1^{-1}T_2 E_1T_2T_3T_1^{-1}T_2^{-1} E_1 T_4T_5T_3^{-1} T_4^{-1} \tilde E_1E_3\\
 &=qz^{-1}T_1^{-1}T_2 E_1T_2T_3T_4T_5 T_1^{-1}T_2^{-1}T_3^{-1} T_4^{-1}  E^2\\& =qz^{-1}T_1^{-1}T_2 E^3,
 \\
 \end{aligned}  $$
 where the last equality follows from  Definition~\ref{k2222}(1).
 So, $$ \phi( E^2T_2T_3T_1^{-1}T_2^{-1})=qz^{-1} T_1^{-1}T_2 E^3  T_4T_5T_3^{-1}T_4^{-1}  =qz^{-1} T_1^{-1} T_2 E^3=\phi(T_2T_3T_1^{-1}T_2^{-1}E^{2}).$$
  This proves that the images of  generators also satisfy Definition~\ref{qqq}(5) and hence  $\phi$ is an algebra homomorphism.
When  $i=0, 1$, (2) follows from the definition of $\phi$.
  In general, by induction assumption on $i-1$ we have
 $$\begin{aligned}
 	\phi(E^i) & =\phi(E_1 T_{2, 2i} T_{2i-1, 1}^{-1} E^{i-1})                                                        \\
 	          & =qz^{-1} \phi(E_1 T_{2, 2i} T_{2i-1, 1}^{-1})
  T_1^{-1} T_2 E^i                                     \\
 	          & = q^2 z^{-2} T_1^{-1} T_2 E_1T_2T_3T_1^{-1}T_2^{-1}E_1 T_{4, 2i+2} T_{2i+1, 3}^{-1}T_1^{-1} T_2 E^i  \\
 	          & = q^2 z^{-2} T_1^{-1} T_2 E_1T_2T_3T_1^{-1}T_2^{-1} T_{4, 2i+2} T_{2i+1, 3}^{-1}E_1 T_1^{-1} T_2 E^i \\
 	          & =q z^{-1} T_1^{-1} T_2 E_1 T_{2, 2i+2} T_{2i+1, 1}^{-1}  E^i                                         \\
 	          & =qz^{-1} T_1^{-1} T_2 E^{i+1},
 \end{aligned}
 $$
where the last equality follows from  Definition~\ref{k2222}(1). This  proves (2).  \end{proof}

\section{The classical branching rule }
In subsection~3.1, we recall some well-known results on Hecke algebras.  In the remaining part of this section, we study the classical branching rule for the $q$-Brauer algebra.
\subsection{ Cell filtration of cell  modules and  permutation modules for Hecke algebras}
 For any $0\le f\le \lfloor n/2\rfloor$, let ${\mathfrak S}_{2f+1, n}$ be the symmetric group on letters $\{2f+1, 2f+2, \ldots, n\}$. When $n$ is even and $f=n/2$, we set  $
\mathfrak S_{2f+1, n}=1$. Let  $\mathcal H_{2f+1, n}$ be the Hecke algebra associated to $\mathfrak S_{2f+1, n}$. Then $\mathcal H_{n-2f}\cong  \mathcal H_{2f+1, n} $. The corresponding isomorphism  sends $T_i$ to  $T_{2f+i}$, $1\le i\le n-2f-1$.

Recall that a composition  $\lambda$ of a non-nagative integer  $d$ is a sequence $(\lambda_1, \lambda_2, \cdots )$ of non-negative integers such that $|\lambda|:=\sum_{i\ge 1} \lambda_i=d$. If $\lambda_i\ge \lambda_{i+1}$ for all possible $i$, then $\lambda$ is called a partition.  Given a positive integer $e$, a partition $\lambda$ of $d$ is called $e$-restricted if $\lambda_i-\lambda_{i+1}<e$ for all possible $i$. When $e>d$, any partition of $d$ is $e$-restricted. Let $\Lambda(d)$ (resp., $\Lambda^+(d)$, resp.,   $\Lambda^+_e(d)$) be the set of all compositions (resp., partitions, resp.,  $e$-restricted partitions) of $d$, where $e$ is always the quantum characteristic of $q^2$. Then $e$ is the minimal positive integer such that
 $$1+q^2+\cdots+q^{2e-2}=0.$$ If such a positive integer does not exist, then $e=\infty$. In this case,  $q^2$ is not a root of unity.
 It is known that each of $\Lambda(d)$, $\Lambda^+(d)$ and    $\Lambda^+_e(d)$ is a poset with dominance order $\unlhd$ defined on it such that $\mu\unlhd\lambda$ if
 $\sum_{j=1}^i \mu_j\le \sum_{j=1}^i \lambda_j$ for all possible $i$. If $\mu\unlhd \lambda$ and $\mu\neq \lambda$, we write $\mu\lhd \lambda$.

Each composition  $\lambda$ of $n-2f$ corresponds to the Young diagram $[\lambda]$ such that there are $\lambda_i$ boxes in the $i$-th row of $[\lambda]$. A $\lambda$-tableau $\t$ is obtained from $[\lambda]$ by inserting $2f+1, 2f+2, \ldots, n$ into $[\lambda]$ without repetition. In this case, we write $\text{shape}(\t)=\lambda$ and call $\lambda$ the shape of $\t$.
  If the entries in $\t$ increase from left to right along row and down column, $\t$ is called standard. In this case, $\lambda\in \Lambda^+(n-2f)$.
Let $\Std(\lambda)$ be the set of all standard $\lambda$-tableaux. Then  $\t^\lambda\in \Std(\lambda)$, where $\t^\lambda$   is  obtained from $[\lambda]$ by inserting $2f+1, 2f+2, \ldots, n$ successively from left to right along the rows of $[\lambda]$. For any $\lambda\in \Lambda^+(n-2f)$,  let $\mathfrak S_\lambda$ be the Young subgroup with respect to $\lambda$. Then $\mathfrak S_\lambda$ is the subgroup of $\mathfrak S_{2f+1, n}$ which stabilizes the entries in each row of $\t^\lambda$. For example, $\mathfrak S_\lambda$ is the subgroup of $\mathfrak S_{3,9}$ generated by $s_3,s_4,s_6$ and $s_8$  if $n=9, f=1$ and $\lambda=(3,2,2)$. In this case,
$$
\t^{\lambda}= \young(345,67,89).$$
For each  $\t\in \Std(\lambda)$, there is a distinguished right coset representative $d(\t)$ of $\mathfrak S_\lambda$ in ${\mathfrak S}_{2f+1, n}$ such that
   $\t=\t^\lambda d(\t)$. Suppose $\s, \t\in \Std(\lambda)$.  Following \cite{Ma}, define $$x_{\s\t}=T_{d(\s)}^*x_\lambda T_{d(\t)},$$ where  $x_\lambda=\sum_{w\in {\mathfrak S}_\lambda}q^{\ell(w)}T_w$ and $*$ is the anti-involution on $\mathcal H_{2f+1, n}$ fixing all generators $T_j$'s.

  % For any $\lambda\in \Lambda(n-2f) $, let   $$ where ${\mathfrak S}_\lambda$ is the Young subgroup of ${\mathfrak S}_{2f+1, n}$  with respect to $\lambda$.

%The following result gives  the Jucys-Murphy basis of $ H_{2f+1, n}$.
\begin{Theorem}\label{mur}\cite[Theorem~3.20]{Ma} The Hecke algebra  $\mathcal  H_{2f+1, n}$ is free over  $\mathcal Z$ with  basis
$$\{ x_{\s\t}\mid  \s,\t\in\Std(\lambda),  \lambda\in\Lambda^+(n-2f)  \}.$$
It is a cellular basis in the sense of \cite[Definition~1.1]{GL}. The required anti-involution is $\ast$ as above. \end{Theorem}

Suppose  $(\lambda, \mu)\in\Lambda^+(n-2f)\times \Lambda^+(n-2f-1)$. The partition $\mu$ is obtained from $\lambda$ by removing a  removable node, say  $p=(k, \lambda_k)$ of $\lambda$ if $\mu_j=\lambda_j$, $j\neq k$ and $\mu_{k}=\lambda_k-1$. In this case, $\lambda$ is obtained from $\mu$ by adding the addable node $p$ of $\mu$.  We write either $\lambda=\mu\cup p$ or $\mu=\lambda\setminus p$.
Let $\mathcal R_\lambda$ be the set of all
partitions obtained from $\lambda$ by removing a removable node and $\mathcal A_\lambda$ the set of all
partitions obtained from $\lambda$ by adding an addable node.

For any  $\lambda\in \Lambda^+(n-2f)$, let  $S^{\lambda}$ be  the cell module of   $\mathcal H_{2f+1, n}$ with respect to the Jucys-Murphy basis in Theorem~\ref{mur}.  By \cite[Definition~2.1]{GL}    $S^{\lambda}$  can be identified with  the free $\mathcal Z$-module with basis
$\{x_{\t}\mid \t\in \Std(\lambda)\}$ where $$x_\t:=x_{\t^\lambda \t}+\mathcal  H_{2f+1, n}^{\rhd \lambda}$$ and $\mathcal  H_{2f+1, n}^{\rhd \lambda}$ is the free $\mathcal Z$-submodule  of $\mathcal H_{2f+1, n}$ spanned by $\{x_{\u \s} \mid \u, \s \in \Std(\mu), \mu\rhd \lambda\}$. It is also a two-sided ideal of $\mathcal H_{2f+1, n}$. Suppose $n>2f$.
Write   $\mathcal R_\lambda=\{\mu^{(i)}\mid 1\le i\le a\}$ for some positive integer $a$ such that
 $$\mu^{(1)}\rhd \mu^{(2)}\rhd \cdots\rhd \mu^{(a)}.$$ For any standard $\lambda$-tableau $\t$, let
$\t\!\!\!\downarrow_{n-1}$ be obtained from  $\t$ by removing the entry $n$. Then $\t\!\!\!\downarrow_{n-1}\in\Std(\mu)$ for some  $\mu\in \mathcal R_\lambda$. Let $S^\lambda_i= \mathcal Z\text{--span}\{x_\t\mid  \t\!\!\!\downarrow_{n-1}\in \Std(\mu^{(j)}), 1\le j\le i\}$. Then $S^\lambda_i$ is a right $ \mathcal H_{2f+1, n-1}$-module such that $$S^{\lambda}=S^\lambda_a\supset S^\lambda_{a-1}\supset \cdots \supset S^\lambda_1\supset  S^\lambda_0=0.$$

 \begin{Theorem}\label{hecbr}\cite[Proposition~6.1]{Ma} As $\mathcal  H_{2f+1, n-1}$-modules, $ S^\lambda_i/S^\lambda_{i-1}\cong S^{\mu^{(i)}}, 1\leq i\leq a$,
where $ S^{\mu^{(i)}}$ is the cell module  with respect to the Jucys-Murphy basis of
 $\mathcal  H_{2f+1, n-1}$ in Theorem~\ref{mur}. \end{Theorem}

Suppose $(\lambda,\mu) \in\Lambda(n-2f)\times \Lambda(n-2f)$. Recall that a $\lambda$-tableau of type $\mu$ is obtained from $[\lambda]$ by inserting integers  $2f+i$ into
   $[\lambda]$ such that $2f+i$ appears $\mu_i$ times.  A $\lambda$-tableau $S$ of type $\mu$ is called row semi-standard if the  entries in each row of  $S$ are
   non-decreasing from left to right. A row semi-standard tableau
    $S$ is called semi-standard if $\lambda$ is a partition and the entries
    in each column of $S$ are strictly increasing downward.
    Let $\mathcal T^{ss}(\lambda,\mu)$ be the set of all semi-standard $\lambda$-tableaux of type $\mu$. When $\mu=(1, 1, \ldots, 1)\in \Lambda^+(n-2f)$,  $\mathcal T^{ss}(\lambda,\mu)$ is $\Std(\lambda)$.

    For any $\t\in\Std(\lambda)$, let $\mu(\t) $ be the  $\lambda$-tableau obtained from $\t$ by replacing each entry $i$ in $\t$ with $2f+k$ if $i$ appears in the $k$-th row of $\t^{\mu}$.
 By~\cite[Example~4.2(ii)]{Ma},  $\mu(\t)$ is a row semi-standard $\lambda$-tableau of type $\mu$.
Following \cite[Chapter~4, \S2]{Ma}, write \begin{equation}\label{sst} x_{S,\t}=\sum_{\s\in\Std(\lambda), \mu(\s)=S}q^{\ell(d(\s))}x_{\s,\t}\end{equation}
where  $(S, \t)\in \mathcal T^{ss}(\lambda,\mu)\times\Std(\lambda)$.

\begin{Lemma}\label{matsho}\cite[Corollary~4.10]{Ma} For any $\mu\in \Lambda(n-2f)$,
 the right $\mathcal H_{2f+1, n}$-module $x_\mu \mathcal  H_{2f+1, n}$  has basis
 $\{x_{S,\t}\mid S\in \mathcal T^{ss}(\lambda,\mu),\t\in\Std(\lambda), \lambda\in\Lambda^+(n-2f) \}$.
  Arrange all such semi-standard tableaux as $S_1, S_2, \cdots, S_k$ such that $S_i\in  \mathcal T^{ss}(\lambda^{(i)},\mu)$ and $i>j$
  whenever $\lambda^{(i)}\rhd \lambda^{(j)}$. Let $M_i$ be the $\mathcal Z$-submodule of $x_\mu \mathcal  H_{2f+1, n}$ spanned by
 $ \{x_{S_j,\t}\mid  \t\in \Std(\lambda^{(j)}), i\le j\}$. Then  $x_\mu \mathcal  H_{2f+1, n}$
  has a filtration $$x_\mu \mathcal  H_{2f+1, n}=M_1\supset M_2\supset \ldots\supset M_k\supset M_{k+1}= 0$$ such that $M_i/M_{i+1}\cong S^{\lambda^{(i)}}$. The required isomorphism sends   $ x_{S_i,\t}+M_{i+1} $ to $x_{\lambda^{(i)}}T_{d(\t)} +\mathcal H^{\rhd \lambda^{(i)}}_{2f+1, n} $.
\end{Lemma}

\subsection{A cellular basis of the $q$-Brauer algebra  }In this subsection, we assume that $\mathcal B_n(q, z)$ is defined over $\Bbbk$.
Let  $$\Lambda_n=\left\{(f,\lambda)\mid \lambda\in \Lambda^+(n-2f), 0\leq f\leq \lfloor \frac{n}{2}\rfloor\right\} .$$ There is a partial order $\unlhd$ on  $ \Lambda_n$  such that $ (\ell, \mu)\unlhd (f, \lambda) $ if $\ell <f $ or $\ell=f$ and $ \mu\unlhd \lambda$. Write $(\ell, \mu)\lhd (f, \lambda)$
if  $(\ell, \mu)\unlhd (f, \lambda) $ and $(\ell, \mu)\neq  (f, \lambda) $.
Later on, we identify each $(\ell, \mu)$ in $\Lambda_n$ with $\mu$. So $\mu\unlhd \lambda$ if
$|\mu|>|\lambda|$ or $|\mu|=|\lambda|$ and $\mu\unlhd \lambda$.

\begin{Defn}\label{dlf}~\cite[Lemma~4.3]{RX},~\cite[(1.5)]{Wen1} Suppose $0\le f\le \lfloor \frac n2\rfloor$. Define
\begin{itemize}\item[(1)] $\mathcal D_{f,n} =\{s_{2f,i_f}s_{2f-1,j_f}\cdots s_{2,i_1}s_{1,j_1}\mid  i_1<\ldots<i_f, 2k-1\leq j_k<i_k\leq n, \text{ for }1\leq k\leq f \}$, \item[(2)]  $B_{f, n}=\{ t_{n-1}t_{n-2}\cdots t_{2f} t_{2f-2}  \cdots t_2\mid t_j=1 \text{ or }
 t_{j}=s_{i_j, j+1}, 1\le i_j\le j+1\}$.
\end{itemize}
 \end{Defn}
 When $f=0$, the RHS of the  equality in  Definition~\ref{dlf}(1) is the empty set and we set $\mathcal D_{0, n}=\{1\}$.

\begin{Lemma}\label{he1}  Suppose  $d^{-1}\in \mathcal D_{f, n-1}$ and  $h\in \{s_{n-1}, s_{n-2, n}\}$. We have
$dh=s_{l, n}  d_1$ for some integer $1\le l\le n$ and   $d_1^{-1}\in \mathcal D_{f, n-2}$.
% where $t_{n-1}$ is given in
%\eqref{bfn}.
\end{Lemma}
\begin{proof} Suppose $2f-1\le j_f<i_f\le n-1$ and $2f<k\le n-1$. We have
 \begin{equation}\label{fff123} s_{j_f, 2f-1} s_{i_f, 2f}  s_{k,n}=\begin{cases} s_{k,n} s_{j_f, 2f-1} s_{i_f, 2f}   &\text{if $i_f<k$, }\\
s_{k-1,n} s_{j_f,2f-1} s_{i_f-1, 2f} &\text{if $i_f\ge k, j_f<k-1$,}\\
s_{k-2, n} s_{j_f-1, 2f-1} s_{i_f-1, 2f} &\text{if $i_f\ge k, j_f\ge k-1$.}\\
\end{cases}\end{equation}
Now  the result follows if we use \eqref{fff123} $f$ times to rewrite
$dh$ successively.  \end{proof}

\begin{Lemma}\label{wen321} For any $0\le f\le \lfloor \frac n2\rfloor$, $\mathcal D^{-1}_{f, n}\subset B_{f, n}$, where
$\mathcal D^{-1}_{f, n}=\{d\mid d^{-1}\in \mathcal D_{f, n}\}$.
\end{Lemma}
\begin{proof} The result for $n=2$ is trivial.  In general, write
$d= d_{i_f, j_f} d_{i_{f-1}, j_{f-1}}\cdots d_{i_1, j_1}$ for any $d\in
 \mathcal D_{f, n}$, where
$$d_{i_f, j_f}=  s_{2f, i_f} s_{2f-1, j_f}.$$
Suppose $n=2f$. Thanks to Definition~\ref{dlf}(1),  $(i_f, j_f)=(n, n-1)$ and  $$d^{-1}\in \mathcal D^{-1}_{f-1, n-1}\subset B_{f-1, n-1}\subset B_{f, n},$$
where the first inclusion follows from induction assumption
 on $n-1$.  Suppose $n>2f$. If $i_f<n$ then
$d^{-1}\in \mathcal D^{-1}_{f, n-1}$  and hence $d^{-1}\in B_{f, n}$ by induction assumption on $n-1$. If $i_f=n$, then  there is a $d_1\in \mathcal D_{f-1, n-1}$
such that
\begin{equation}\label{wen123} d^{-1}=d_1 s_{j_f, 2f-1} s_{i_f, 2f}=\begin{cases}
d_1 s_{n-1} s_{j_f, 2f-1} s_{n-1, 2f} & \text{ if $j_f<n-1$,}\\
d_1s_{n-2} s_{n-1} s_{n-2, 2f-1} s_{n-1, 2f} & \text{ if $j_f=n-1$.}
\\
\end{cases}\end{equation}
We use  Lemma~\ref{he1}  to rewrite both $d_1 s_{n-1}$ and $d_1s_{n-2} s_{n-1}$
in \eqref{wen123}. So there is a $d_1'\in \mathcal D^{-1}_{f, n-1}$ such that
 $d^{-1}=t_{n-1} d_1' $.  Now, the result follows from induction assumption on $n-1$.
\end{proof}

 Recall that a Brauer diagram
is a diagram with $2n$ vertices arranged in two rows and $n$ edges such that each vertex belongs to a unique edge. The vertices in
the top (resp., bottom) row are labeled as $n+1, n+2, \ldots, 2n$ (resp., $1, 2, \ldots, n$) from left to right.
Then each Brauer diagram $d$  corresponds to a unique partition of $1, 2, \cdots, 2n$ into $n$ pairs
$\{ (i_k, j_k)\mid 1\le k\le n, i_k<j_k\}$.
We denote $\{ (i_k, j_k)\mid 1\le k\le n\}$   by $\text{conn}(d)$ and  call it  the
connector of $d$. An edge is called a horizontal edge if two vertices of it are at the same row. Otherwise, it is
called a vertical edge.  Let $d_f$ be the Brauer diagram such that
$$\text{conn}(d_f)=\{(2i-1, 2i), (n+2i-1, n+2i)\mid 1\le i\le f\}\cup \{(j, n+j)\mid 2f+1\le j\le n\}.$$
Consider a Brauer diagram on which there are exactly $f$ horizontal edges
$(n+2i-1, n+2i)$, $ 1\le i\le f$,
at the top row and there is no crossing
between any two vertical edges.  Let  $D_{f, n}$ be the set of all such Brauer diagrams. The symmetric group $\mathfrak S_n$ acts transitively on the right of  $D_{f, n}$. In \cite[(2.11)]{N21}, Dung Tien Nguyen
defined $\mathscr B_{0, n}=\{1\}$ and
\begin{equation}\label{N22} \mathscr B_{f, n}=\{w^{-1}\in B_{f, n} \mid d=d_f w\in  D_{f, n}\}\end{equation}
for $f>0$.
There is a  restriction  on the length of $d$ in \cite[(2.11)]{N21}. However, this restriction is  redundant since $d$ automatically satisfies this
condition if $w^{-1}\in B_{f,n}$~\cite[Lemma~2.1]{Wen1}.

\begin{Prop}\label{equal} For any $0\le f\le \lfloor \frac n 2\rfloor$,
$\mathcal D_{f, n}=\mathscr B_{f, n}$.
\end{Prop}
\begin{proof} There is a well-known result for Brauer diagrams which says that  $d_f w\in D_{f, n}$ if $w\in \mathcal D_{f, n}$.   Thanks to
Lemma~\ref{wen321} and \eqref{N22}, $\mathcal D_{f, n}\subseteq \mathscr B_{f, n}$. We have $\mathcal D_{f, n}=\mathscr B_{f, n}$ since
the cardinalities of  $\mathcal D_{f, n}$ and $\mathscr B_{f, n}$ are
$\frac{n!}{2^f(n-2f)!f!}$ (see \cite[Lemma~4.3]{RX} and \cite[Remark~3.18]{N18}).
\end{proof}

Thanks to Proposition~\ref{equal}, we can   use $\mathcal D_{f, n}$ to replace $\mathscr B_{f, n}$ when we state
the  cellular basis of $\mathcal B_n$ in \cite[Theorem~3.2]{N21}.
This  has an advantage  when we  do some explicit computation later on.

%\begin{Defn}\label{cst} Suppose   $(f, \lambda)\in \Lambda_n$.
%For any $(w, \s), (v, \t)\in I(\lambda):=\mathcal D_{f, n}\times \Std(\lambda)$, define
% $ C_{(w,\s),(v,\t)}^\lambda= \sigma (T_{w})E^f x_{\s,\t} T_v$,
% where  $\sigma$ is the anti-involution  in Lemma~\ref{anti}  and  $E^f=E_1E_3\cdots E_{2f-1}$ (see Lemma~\ref{k3}(4)). \end{Defn}

\begin{Theorem}\label{sstan}  \cite[Theorem~3.2]{N21} The $q$-Brauer algebra
 $\mathcal B_n$ is free over $\Bbbk$ with cellular basis
 $$S=\{ C_{(w,\s),(v,\t)}^\lambda \mid  (w, \s), (v, \t)\in I(\lambda), (f, \lambda)\in \Lambda_n\}$$
 in the sense of \cite[Definition~1.1]{GL}, where   $ I(\lambda):=\mathcal D_{f, n}\times \Std(\lambda)$ and $ C_{(w,\s),(v,\t)}^\lambda= \sigma (T_{w})E^f x_{\s\t} T_v$.
 The required anti-involution is $\sigma$ in Lemma~\ref{anti}.
\end{Theorem}

 \begin{Lemma}\label{efb}\cite[Corollary~3.1]{N21}
Suppose $0\le f\le \lfloor \frac n 2\rfloor$. Then $E^f \mathcal B_n +\mathcal B_n^{f+1}$ is a left $\mathcal H_{2f+1, n}$-module spanned by $\{E^f T_d+\mathcal B_n^{f+1}\mid d\in \mathcal D_{f, n}\}$, where
$\mathcal B_n^{f+1} $ is the two-sided ideal of $\mathcal B_n$ generated by $E^{f+1}$.\end{Lemma}

\begin{Theorem}\label{iso111}Let
$\phi: \mathcal B_{n-2}\rightarrow \tilde E_1 \mathcal B_{n} \tilde E_1$ be the algebra homomorphism
in Proposition~\ref{ffunc1}. Then $\phi$ is an algebra isomorphism.   \end{Theorem}
\begin{proof}  Thanks to Theorem~\ref{sstan}, $\mathcal B_{n-2}$ has basis  $$\{\sigma(T_{d_1}) E^f T_w T_{d_2}\mid
d_1, d_2\in \mathcal D_{f, n-2}, w\in \mathcal H_{2f+1, n-2}, 0\le f\le \lfloor \frac{ n-2}{2}\rfloor\}.$$
By Proposition~\ref{ffunc1}(2), $\phi(\sigma(T_{d_1})
 E^f T_w T_{d_2})=qz^{-1} T_1^{-1}T_2\sigma(T_{e_1}) E^{f+1}  T_{w'} T_{e_2}$ where $e_1, e_2, w'$ are obtained from $d_1, d_2, w$ by replacing each factor  $s_i$ in  $d_1, d_2, w$ by $s_{i+2}$.
 This proves that $\phi(\sigma(T_{d_1})
 E^f T_w T_{d_2})$ is a basis element of $\mathcal B_n$, and hence $\phi$ is a monomorphism.
 On the other hand, there are some scalars $a, b\in \Bbbk$ and $w_1, w_2\in \mathfrak S_{3, n}$ such that \begin{equation}\label{aast} E_1 T_{2, i_1} T_{1, j_1}  E_1= aE_1E_3T_{w_1} +bE_1 T_{w_2}\end{equation}  and $a=0$ unless
$i_1\ge 4$ and $j_1\ge 3$.  In general, since $\mathcal D_{1, n}$ is a right coset representatives of $\mathfrak S_2\times \mathfrak S_{3, n}$, by Definition~\ref{qqq}(2)(3) and
\eqref{aast},
$E_1\mathcal H_n E_1=E_1E_3 \mathcal H_{3, n}+E_1 \mathcal H_{3, n}$.
 By  Theorem~\ref{sstan} and \eqref{aast},
 $E_1 \mathcal  B_nE_1 $,
 is generated by $ E_1E_3$, $E_1T_j$, $3\leq j\leq n-1$.
 Consequently,  $\tilde E_1\mathcal B_{n} \tilde E_1$ is generated by
 $ \tilde E_1E_3$, $\tilde E_1T_j$, $3\leq j\leq n-1$ and hence
  $\phi$ is an epimorphism.
\end{proof}

%\label{imp}
 \begin{Cor}\label{e1be1} Suppose  $n\ge 2$.
 \begin{itemize}\item [(1)]  There is an algebra isomorphism $\psi:
  \mathcal B_{n-2}\rightarrow E_1 \mathcal B_n E_1$ sending
  $1, E_1$ and $ T_i$ to $\delta^{-1} E_1, \delta^{-1} E_1E_3$ and $ \delta^{-1} E_1T_{i+2}$
  for all $1\le i\le n-3$.
  \item[(2)] The right $\mathcal B_n$-module $ E^f \mathcal B_{n}$ is a left $ E^f \mathcal B_{n} E^f$-module generated by
$ E^f T_d$, $d\in \mathcal D_{f, n}$.
Further, $E^f \mathcal B_n E^f $ is generated by $E^{f+1}$ and $ E^f T_{2f+i}$, $1\le i\le n-2f-1$.
\end{itemize}
  \end{Cor}
 \begin{proof}  (1) follows from  arguments in the proof of Theorem~\ref{iso111} and
  (2) follows from \eqref{aast} and Lemma~\ref{efb} successively for $\mathcal B_n$, $\mathcal B_{n-2}$ etc.  \end{proof}

\subsection{The classical branching rule} In this subsection, we go on assuming  that $\mathcal B_n$ is defined over $\Bbbk$. For any $(f, \lambda)\in \Lambda_n$, let $C(f, \lambda)$ be the (right) cell module of $\mathcal B_n$
 with respect to the cellular basis in  Theorem~\ref{sstan}.
  Up to an isomorphism, $C(f, \lambda)$  can be considered as  the free $\Bbbk$-module with basis
$\{E^f x_\lambda T_{d(\t)}T_v +\mathcal B_n^{\rhd \lambda}\mid  (v, \t)\in I(\lambda)\} $, where
$$ \mathcal B_n^{\rhd \lambda}=\Bbbk\text{-span}\{ C_{(w,\s),(v,\t)}^\mu \mid  (w, \s), (v, \t)\in I(\mu), (\ell, \mu)\in \Lambda_n, \mu \rhd  \lambda\}.$$

 %For any  $(f, \lambda)\in \Lambda_n)$,
%let \begin{equation}\label{arrowlambda}
%\mathcal{RA}(\lambda)=  \{\mu\in \mathcal A_\lambda\mid (f-1, \mu)\in \Lambda_{n-1}\}\cup
%\{\mu\in  \mathcal R_\lambda\mid (f, \mu)\in \Lambda_{n-1}  ).\end{equation}

%We write $\mu\rightarrow \lambda$ if $\mu\in \mathcal{RA}(\lambda)$.

\begin{Defn}\label{deofymulambda}
 For any  $(f, \lambda)\in\Lambda_n$, let
 $
\mathcal{RA}(\lambda)$ be the set of   all partitions $\mu\in \mathcal A_\lambda\cup \mathcal R_\lambda$ such that  $(l, \mu)\in \Lambda_{n-1}$ for some $l\in \mathbb N$.
Let
$$y_{\mu}^\lambda=\begin{cases}  E^fx_\lambda T_{a_k,n}
&  \text{if $\mu=\lambda\setminus(k,\lambda_k)$,}\\
E_{2f-1}T_{n,2f}^{-1}T_{b_k,2f-1}^{-1}E^{f-1}x_{\mu} & \text{if  $\mu=\lambda\cup (k,\lambda_k+1)$,}\\
\end{cases}$$where $a_k=2f+\sum_{i=1}^k\lambda_i$ and $b_k=2f-1+\sum_{i=1}^k\lambda_i$.
\end{Defn}

Thanks to the defining relations for $\mathcal B_n$, we can rewrite  $y_{\mu}^\lambda$ as follows:
\begin{equation}\label{symulambda}
y_{\mu}^\lambda=\begin{cases}
                        \sum_{i=a_{k-1}+1}^{a_k}q^{a_k-i}T_{i,n}E^f x_\mu & \hbox{if $\mu=\lambda\setminus (k,\lambda_k)$,} \\
                        E^fx_\lambda T_{n,2f}^{-1}T_{b_k,2f-1}^{-1}\sum _{i=b_{k-1}+1}^{b_k}q^{b_k-i}T_{b_k,i} & \hbox{if $\mu=\lambda\cup (k,\lambda_k+1) $.}\\
                        \end{cases}
\end{equation}
For any $(f, \lambda)\in\Lambda_n$, there is a pair   $(a, m)\in \mathbb N\times \mathbb N\setminus \{0\}$  such that
$$\mathcal{RA}(\lambda)
=\{\mu^{(1)},\mu^{(2)},\ldots, \mu^{(a)}\}\cup \{\mu^{(a+1)},\mu^{(a+2)},\ldots, \mu^{(m)}\}$$
and $(f, \mu^{(i)})\in \Lambda_{n-1}, 1\le i\le a$ and $(f-1, \mu^{(i)})\in \Lambda_{n-1}, a+1\le i\le m$.  Obviously $\mu^{(a+1)}$ occurs only if $f>0$. We can arrange them so that
 \begin{equation}\label{ppp11} \mu^{(1)} \rhd  \mu^{(2)}\rhd \ldots \rhd \mu^{(a)}\rhd\mu^{(a+1)}\rhd
  \mu^{(a+2)}\rhd \ldots \rhd \mu^{(m)}
 \end{equation}
with respect the partial order $\rhd$ on $\Lambda_{n-1}$.
Note that $\mu^{(m)}=\lambda\cup (k+1, 1)$ if $k=\max\{j\mid \lambda_j\neq 0\}$. So
\begin{equation}\label{ymum}y^\lambda_{\mu^{(m)}}=E^f x_\lambda  T_{n, 2f}^{-1} T_{n-1, 2f-1}^{-1}. \end{equation}

\begin{Defn}\label{nkk} For any $1\leq k\leq m$, let $
 N_k=\sum_{j=1}^ky_{\mu^{(j)}}^\lambda \mathcal B_{n-1} + \mathcal B_n^{\rhd\lambda}$.
% where $y_{\mu^{(j)}}^\lambda$'s are given in Definition~\ref{deofymulambda}.
 \end{Defn}
It follows from Definition~\ref{nkk} that $N_k, 1\le k\le m$ are   $\mathcal B_{n-1}$-submodules of $C(f, \lambda)$.

\begin{Lemma}\label{fiels1}Suppose $1\leq k\leq a$.

 \begin{itemize}\item[(1)] The $\mathcal B_{n-1}$-submodule  $N_k$ of $C(f, \lambda)$   is spanned by  $$\{y_{\mu^{(j)}}^\lambda T_{d(\u)}T_w+ \mathcal B_n^{\rhd \lambda}\mid  \u\in\Std(\mu^{(j)}), w\in \mathcal D_{f,n-1}, 1\leq j\leq k\}.$$
 \item[(2)]As $\mathcal B_{n-1}$-modules,  $N_k/N_{k-1}\cong C(f, \mu^{(k)}) $,
 \item[(3)] The submodule  $N_k$  has basis $$\{E^f x_\lambda T_{d(\s)}T_v+\mathcal B_n^{\rhd \lambda}\mid \s\in
\Std(\lambda), \text{ shape}(\s\!\downarrow_{n-1})
\unrhd \mu^{(k)}, v\in \mathcal D_{f,n-1} \}.$$
\end{itemize}
 \end{Lemma}

\begin{proof} Suppose $b\in \mathcal B_{n-1}$ and $\mu^{(k)}=\lambda\setminus (h, \lambda_h)$. Let $a_k=2f+\sum_{i=1}^h \lambda_i$.
 Thanks to  Definition~\ref{nkk}, Lemma~\ref{efb} and Theorem~\ref{hecbr},
$$ y_{\mu^{(k)}}^\lambda b= x_\lambda T_{a_k,n}\sum_{v\in \mathfrak S_{3,n}, w\in \mathcal D_{f,n-1}} a_{v,w}T_vE^fT_w\equiv \sum_{ \u\in\Std(\mu^{(k)}), w\in \mathcal D_{f,n-1}}b_{\u,w} y_{\mu^{(k)}}^\lambda T_{d(\u)}T_w \mod N_{k-1} $$
 where $a_{v,w}$ and $b_{\u,w}$ are some scalars. This proves (1).
  Suppose
$(\u, w)\in \Std(\mu^{(k)})\times  \mathcal D_{f,n-1}$. Then $$y_{\mu^{(k)}}^\lambda T_{d(\u)}T_w+ \mathcal B_n^{\rhd \lambda}
 =E^fx_\lambda T_{d(\t)}T_w+\mathcal B_n^{\rhd \lambda}$$ with
 $\t\!\downarrow_{n-1}=\u$. Since the RHS of the above equality is  a basis element of $C(f, \lambda)$,
by (1),
 $N_k/N_{k-1}$ is  free over $\Bbbk$ with basis $M_k$, where
   \begin{equation}\label{bbasis} M_k=\{y_{\mu^{(k)}}^\lambda T_{d(\u)}T_w+N_{k-1}\mid  \u\in\Std (\mu^{(k)}), w\in \mathcal D_{f,n-1}\}.\end{equation}
 There is  a well-defined   $\Bbbk $-linear isomorphism
   $\phi: N_k/N_{k-1}\rightarrow  C(f, \mu^{(k)}) $ such that
\begin{equation}\label{nkk123} \phi(y_{\mu^{(k)}}^\lambda T_{d(\u)}T_w+N_{k-1})= E^f x_{\mu^{(k)}}T_{d(\u)}T_w +
\mathcal B_{n-1}^{\rhd\mu^{(k)}},\end{equation}   for all  $(\u, w)\in\Std(\mu^{(k)})\times \mathcal D_{f,n-1}$. It follows immediately from
Lemma~\ref{efb}  and Theorem~\ref{hecbr}  that $\phi$ is a right $\mathcal B_{n-1}$-homomorphism and hence $\phi$ is a  $\mathcal B_{n-1}$-isomorphism. This completes the proof of (2).  Finally, (3) immediately follows from (2), \eqref{nkk123} and induction assumption on $k-1$.
\end{proof}
%If $\lambda\in\Lambda^{+}(n)$ (i.e., $f=0$), then $C(0, \lambda)$ has  a cell  filtration  in Lemma~\ref{fiels1}.
% The following result which gives a characterization of $N_k$  follows from above lemma immediately.
%\begin{Cor}\label{sjdid}For $1\leq k\leq a$,
%$\{E^f x_\lambda T_{d(\s)}T_v+\mathcal B_n^{\rhd \lambda}\mid \s\in
%\Std(\lambda), \text{ shape}(\s\downarrow_{n-1})
%\unrhd \mu^{(k)}, v\in D_{f,n-1} \}$ is a basis of $N_k$.
%\end{Cor}
%\begin{proof} The result follows immediately from Lemma~\ref{fiels1}.\end{proof}

From here to the end of this section, we assume
 $(f, \lambda)\in \Lambda_n$ and  $f>0$. We are going to
   deal with $N_k/N_{k-1}$, $a+1\leq k\leq m$.
%Lemma~\ref{tablsji} result has already been given  in the proof of \cite[Corollary 5.4, Lemma~5.5]{Eng}.
\begin{Lemma}\label{tablsji}\cite{Eng}
Let  $(f-1, \mu)\in\Lambda_{n-1}$ such that  $\mu\rhd \mu^{(m)}$.
\begin{itemize}
\item[(1)]Suppose  $\mu\not\in \{\mu^{(a+1)}, \ldots, \mu^{(m)}\}$. If $(\s,\mu^{(m)}(\s)) \in \Std(\mu) \times\mathcal T(\mu,\mu^{(m)} )$, then
     $\text{shape}(\s\!\downarrow_{n-2})\rhd \lambda$.
\item[(2)] Suppose  $\mu=\mu^{(j)}=\lambda\cup (k,\lambda_k+1)$ for some  $a+1\leq j\leq m$. Let  $\s_j=\t^{\mu^{(j)}}s_{b_j,n-1}$, where $b_j=2f-1+\sum_{i=1}^k\lambda_i$.  Then   $\s_j$ is the unique standard $\mu^{(j)}$-tableau $\s$ satisfying
    $\mu^{(m)}(\s)\in \mathcal T^{ss}(\mu^{(j)},\mu^{(m)} )$. In this case, $\text{shape}(\s\!\downarrow_{n-2})=\lambda$.
\end{itemize}
\end{Lemma}
\begin{proof} The results has already been given in the proof of \cite[Corollary 5.4, Lemma~5.5]{Eng}.\end{proof}

\begin{Lemma}\label{jaijai} Suppose  $(\s,\mu^{(m)}(\s)) \in \Std(\mu)\times \mathcal T^{ss}(\mu,\mu^{(m)} )$ for some $(f-1, \mu)\in \Lambda_{n-1}$. If   $\text{shape}(\s\!\!\!\downarrow_{n-2})\rhd \lambda$, then
  $E^f T_{n,2f}^{-1}T_{n-1,2f-1}^{-1} \sigma (T_{d(\s)})x_{\mu} \in  \mathcal B_n^{ \rhd \lambda}$.
\end{Lemma}
\begin{proof} The arguments in the proof of \cite[Lemma~5.3]{Eng}  depends only on  the braid relations for the  Hecke algebras. Therefore, one can imitate arguments there to verify our current result. We leave details to the reader. \end{proof}

\begin{Lemma}\label{spans}As $\mathcal B_{n-1}$-modules, $C(f, \lambda)=N_m$.
\end{Lemma}
\begin{proof}
Recall that the cell module  $C(f, \lambda)$ has basis $\{E^f x_\lambda T_{d(\t)}T_v +\mathcal B_n^{\rhd \lambda}\mid  (v, \t)\in I(\lambda)\} $.
If $v\in \mathcal D_{f,n-1}$, then \begin{equation}\label{key123}  E^f x_\lambda T_{d(\t)}T_v+\mathcal B_n^{\rhd \lambda}\in E^f x_\lambda T_{d(\t)}\mathcal B_{n-1} +\mathcal B_n^{\rhd \lambda}\subset N_a.\end{equation} The last inclusion follows from Lemma~\ref{fiels1}(3).  Otherwise  $T_v=T_{2f,n}b$ for some $b\in \mathcal H_{n-1}$. So,
\begin{equation}\label{key12} E^f x_\lambda T_{d(\t)}T_v+\mathcal B_n^{\rhd \lambda}\in E^f x_\lambda T_{2f,n}  \mathcal B_n +\mathcal B_n^{\rhd \lambda}
\subset\sum_{j=2f+1}^n  E^f x_\lambda T_{j, n} \mathcal B_{n-1}+ E^f x_\lambda
 T_{n,2f}^{-1}\mathcal B_{n-1}+\mathcal B_n^{\rhd \lambda}.
\end{equation} By  Definitions~\ref{deofymulambda}--\ref{nkk}, \eqref{ymum} and Lemma~\ref{fiels1}(3),
$$E^f x_\lambda
 T_{n,2f}^{-1}\mathcal B_{n-1}+\mathcal B_n^{\rhd \lambda}\in N_m \ \ \text{ and }
 \sum_{j=2f+1}^n  E^f x_\lambda T_{j, n} \mathcal B_{n-1}\in N_a\subset N_m.$$
 This implies
$C(f, \lambda)\subseteq N_m$ and hence $C(f, \lambda)= N_m$ as required.
\end{proof}

%Recall that  $\mathcal B^f_n$ is  the two-sided ideal of $\mathcal B_n$ generated by
%$E^f$. Thanks to  Lemma~\ref{efb},    $\mathcal B^f_n$ has  a basis
%$\{C_{(w,\s),(v,\t)}^\mu \mid  (w, \s), (v, \t)\in I(\mu), (\ell, \mu)\in %\Lambda_n,\ell\geq f \} $.
%The following result follows from   arguments similar to those  for the proof of \cite[Claim~5.7]{Eng}.
\begin{Lemma}\label{dkjhddd}
If $b\in E^{f-1} \mathcal B_{n-1}\cap \mathcal B_{n-1}^f$, then there are $w\in \mathcal D_{f, n-1} $ and
$h_w\in  \mathcal H_{2f+1, n}$ such that
\begin{equation}\label{xjsddcd}
E_{2f-1}T_{n,2f}^{-1}T_{n-1,2f-1}^{-1} b\equiv \sum_{w} h_w E^fT_w \mod \mathcal B_n^{f+1}.
\end{equation}
\end{Lemma}

\begin{proof}
By Theorem~\ref{sstan},  Lemma~\ref{efb} and Corollary~\ref{e1be1}(2), any $b\in E^{f-1} \mathcal B_{n-1}\cap \mathcal B_{n-1}^f$ can be written  as a linear combination of elements in $S$  up to  an element in $\mathcal B_n^{f+1}$,
where
\begin{equation} \label{setS} S=\{T_{j,2f-1}T_{i,2f}E^{f}T_uT_w\mid T_u\in
 \mathcal H_{2f+1, n-1},w\in \mathcal D_{f,n-1}, 2f-1\leq j<i\leq n-1   \}.\end{equation}
So  it is enough to assume $b=
T_{j,2f-1}T_{i,2f}E^{f}T_uT_w$ when we verify \eqref{xjsddcd},.

Suppose  $i=n-1$.  Then $j<n-1$. Let $x= E_{2f-1}T_{n,2f}^{-1}T_{n-1,2f-1}^{-1} T_{j,2f-1}T_{n-1,2f}E^{f}$. By  Lemma~\ref{k3}(2)(5) and Lemma~\ref{k2}(6),
$$\begin{aligned}
 x
& = T_{j+2,2f+1} E_{2f-1}T_{n-1,2f}^{-1}T_{n-1,2f-1}^{-1}T_{n-1,2f}E^{f}\\
 & =T_{j+2,2f+1} E_{2f-1}T_{n,2f}^{-1}T_{2f-1}^{-1} E^{f}\\ &
 =q^{-1}T_{j+2,2f+1} E_{2f-1}T_{n,2f}^{-1}  E^{f}\\
& = q^{-1}T_{j+2,2f+1}E_{2f-1}T_{2f}^{-1}  E^{f}T_{n,2f+1}^{-1}\\ &
=z^{-1}q^{-1} E^{f}T_{j+2,2f+1}T_{n,2f+1}^{-1}.
\end{aligned}$$
and  \eqref{xjsddcd} follows.
Suppose  $i<n-1$. By \eqref{kkk1} and Lemma~\ref{k3}(2)
\begin{equation}\label{p3}
E_{2f-1}T_{n,2f}^{-1}T_{n-1,2f-1}^{-1} T_{j,2f-1}T_{i,2f}E^{f}= T_{j+2,2f+1}T_{i+2,2f+2} y T_{n,2f+2}^{-1}T_{n-1,2f+1}^{-1}
,\end{equation}
where $y=E_{2f-1}T_{ 2f}^{-1}T_{ 2f+1}^{-1} T_{ 2f-1}^{-1}T_{ 2f}^{-1} E^{f}$.
 Thanks to \eqref{kkk1}, $T_i^{-1}=T_i-a$ for any $1\le i\le n-1$, where  $a=q-q^{-1}$. We use it   to rewrite $y$ as follows:
\begin{equation}\label{jswjjd}
y=E_{2f-1}T_{ 2f} T_{ 2f+1}  T_{ 2f-1}^{-1}T_{ 2f}^{-1} E^{f}-(aq^{-1}+
aT_{2f+1}) E_{2f-1}T_{ 2f-1}^{-1}T_{ 2f}^{-1} E^{f}.
 \end{equation}
Since
$E^{f+1}= E_1T_{2,2f+2}T_{2f+1,1}^{-1}E^f$, the first term in the RHS of \eqref{jswjjd} is equal to
\begin{equation}\label{p12} T_{2f-1,1}T_{2,2f}^{-1}
E_1T_{2,2f+2}T_{2f+1,1}^{-1}E^f=T_{2f-1,1}T_{2,2f}^{-1}E^{f+1}\in \mathcal B^{f+1}_n. \end{equation}
Using Lemma~\ref{k3}(8) to rewrite the second term  in the RHS of  \eqref{jswjjd}, we see that
 \eqref{xjsddcd} follows immediately from \eqref{p3}--\eqref{p12}.
\end{proof}

\begin{Lemma}\label{fiels2} As $\mathcal B_{n-1}$-modules, $N_k/N_{k-1}\cong  C(f-1,\mu^{(k)}) $ for  any  $a+1\leq k\leq m$.\end{Lemma}

\begin{proof}  Fix $k$, $a+1\le k\le m$.  We want to compute $y_{\mu^{(k)}}^\lambda h  $ for any
 $h\in \mathcal B_{n-1}$.  Thanks to Theorem~\ref{sstan} and Lemma~\ref{efb},   $E^{f-1}  h+ \mathcal B_{n}^{f+1}$ can be written
  as a linear combination of elements $\bar x:=x+\mathcal B_{n}^{f+1}$, where
  $x\in \{ h_d E^{f-1}T_d, b\}$, and $b\in E^{f-1}\mathcal B_{n-1}\cap \mathcal B_{n-1}^f$, $ d\in \mathcal D_{f-1,n-1}$ and
   $ h_d\in \mathcal H_{2f-1, n-1}$.
 % $ h_d E^{f-1}T_d+ \mathcal B_{n}^{f+1}$ and $b+\mathcal B_{n}^{f+1}$, where
  % $b\in E^{f-1}\mathcal B_{n-1}\cap \mathcal B_{n}^f$, $ d\in \mathcal D_{f-1,n-1}$ and
  % $ h_d\in \mathcal H_{2f+1, n}$.
  So
\begin{equation}\label{sss12345}\begin{aligned}  y_{\mu^{(k)}}^\lambda h & \equiv
\sum_{ d\in \mathcal D_{f-1,n-1}}y_{\mu^{(k)}}^\lambda h_dT_d + E_{2f-1}T_{n,2f}^{-1}
T_{b_k,2f-1}^{-1} x_{\mu^{(k)}}b ~(\text{mod }  \mathcal B_{n}^{f+1})\\
&\equiv \sum_{ d\in \mathcal D_{f-1,n-1}}E^f T_{n,2f}^{-1}
T_{n-1,2f-1}^{-1} x_{\s_k,\t^{\mu^{(k)}}} h_dT_d~\pmod {  N_a},\\
\end{aligned}
\end{equation}where $\s_k$ is given in Lemma~\ref{tablsji}(2).
The second equivalence follows from \eqref{xjsddcd} and Lemma~\ref{fiels1}(3).
Further, $ x_{\s_k,\u}\in x_{\mu^{(m)}} \mathcal H_{2f-1, n-1}
$ for any $\u\in \Std(\mu^{(k)})$.
 Applying Lemma~\ref{matsho} on $x_{\s_k,\t^{\mu^{(k)}}} h_d$, we have
 $$
y_{\mu^{(k)}}^\lambda h
\equiv \sum_{ d\in \mathcal D_{f-1,n-1}}E^f T_{n,2f}^{-1}T_{n-1,2f-1}^{-1}
\sum_{(\s, S)\in\Std(\mu)\times\mathcal T^{ss}(\mu, \mu^{(m)}), \mu\unrhd\mu^{(k)} }a_{S,\s} x_{S,\s} T_d
\pmod {  N_a}.$$
 By   Lemmas~\ref{tablsji}(1) and \ref{jaijai}, the terms with respect to $\mu$ in the RHS of the above equality
    are in $\mathcal B_{n}^{\rhd \lambda}$ if $\mu\not\in \{\mu^{(a+1)}, \mu^{(a+2)}, \ldots, \mu^{(k)}\}$. In the remaining case $S\in \mathcal T^{ss} (\mu^{(j)}, \mu^{(m)})$ for some $a+1\le j\le m$.  Thanks to  Lemma~\ref{tablsji}(2), $\mu^{-1} (S)$ contains a unique element, say $\s_j$ in  $ \Std(\mu^{(j)})$.
  So  $ x_{S,\s}=q^{-\ell (d(\s_j))} x_{\s_j, \s}$ and $d(\s_j)=s_{b_j, n}$, where $b_j$ is given in  Lemma~\ref{tablsji}(2). Since
   $$E^f T_{n, 2f}^{-1} T_{n-1, 2f-1}^{-1} x_{\s_j, \s}=E^f T_{n, 2f}^{-1} T_{n-1, 2f-1}^{-1}
   T_{b_j, n-1} x_{ \s}= y_{\mu^{(j)}}^\lambda  T_{d(\s)},$$
      $N_k/N_{k-1}$ is the $\Bbbk$-module spanned by
\begin{equation}\label{mkspandedd}
M_k =\{ y_{\mu^{(k)}}^\lambda T_{d(\s)} T_d+N_{k-1}\mid  \s\in \Std(\mu^{(k)}),
d\in \mathcal D_{f-1,n-1}\}.\end{equation}
By Lemma~\ref{mkspandedd}, $C(f, \lambda)$ can be spanned by a set whose  cardinality is $\sum_{k=1}^m  |M_k|$. Thanks to \cite[Theorem~2.3]{Wen},  the  rank of cell module for Birman-Murakami-Wenzl algebra with respect to $(f, \lambda)$ is equal to
$\sum _{k=1}^m |M_k|$. Since  $\text{rank} ~ C(f, \lambda)$ is equal to that for Birman-Murakami-Wenzl algebra, we have $\text{rank} ~ C(f, \lambda)=\sum _{k=1}^m |M_k| $ and hence $M_k$ is linear independent for any $1\le k\le m$.
%and hence  $\text{rank}~ N_k/N_{k-1}\leq \text{rank}~ C(f-1,\mu^{(k)})=|M_k|$.
%By the branching rule for the cell module of Birman-Murakami-Wenzl algebras in \cite[Theorem~2.3]{Wen},
%we have $\text{rank} ~ C(f, \lambda)=\sum _{k=1}^m |M_k| $, forcing that
Consequently, % By  Lemma~\ref{spans},
% $\dim N_m=\dim C(f,\lambda)$, forcing  $\dim N_k/N_{k-1}=\dim C(f-1,\mu^{(k)})=|M_k|$.
  $M_k$ is a basis of $N_k/N_{k-1}$ for all $1\le k\le m$.  We have
a well-defined $\Bbbk$-linear isomorphism $\phi: N_k/N_{k-1}\rightarrow C(f-1, \mu^{(k)})$ for any $a+1\le k\le m$ such that
\begin{equation}\label{pppqqq} \phi(y_{\mu^{(k)}}^\lambda T_{d(\s)} T_d+N_{k-1})=E^f x_{\mu^{(k)}}T_{d(\s)} T_d+\mathcal B_{n-1}^{\rhd \mu^{(k)}}\end{equation}
for any  $\s\in \Std(\mu^{(k)})$ and $ d\in \mathcal D_{f-1,n-1}$.  By arguments similar to those above (more explicitly, using  Theorem~\ref{sstan}, Lemmas~\ref{matsho} and \ref{efb})
   we see that    $\phi$   is a right  $\mathcal B_{n-1}$-homomorphism.
\end{proof}

\begin{Theorem}\label{Burnch}Suppose $(f, \lambda)\in \Lambda_n$.
Then there is a filtration
$0=N_0\subset N_1\subset \ldots \subset N_m=C(f, \lambda) $  of $\mathcal B_{n-1}$-modules   such that $$N_k/N_{k-1}\cong \begin{cases} C(f, \mu^{(k)}) &\text{if  $1\leq k\leq a$,} \\
C(f-1, \mu^{(k)}) &\text{if  $a+1\leq k\leq m$.}\end{cases}$$
\end{Theorem}
\begin{proof}The result follows immediately  from  Lemmas~\ref{fiels1} and \ref{fiels2}.\end{proof}

\section{The Jucys-Murphy elements  and Jucys-Murphy  basis }In this section, we construct the Jucys-Murphy elements and the Jucys-Murphy basis of $\mathcal B_n$. Unless otherwise stated, we assume that $\mathcal B_n$ is defined over $\Bbbk$.
\subsection{Jucys-Murphy  elements }  For any $1\le i\le n$, define
\begin{equation}\label{JM1}L_i=\begin{cases} 0 &\text{if $i=1$,}\\
\sum_{j=1}^{i-1} (j, i)-q^2z^{-1}\sum_{j=1}^{i-1} E_{j, i}, &\text{if $2\le i\le n$,} \\
\end{cases}\end{equation}
where
\begin{equation}\label{eij}(j, i)=T_{j, i-1} T_{i-1} T_{i-1,j}\ \text{ and }\ E_{j, i}= T_{1,j}^{-1} T_{i, 2} E_1 T_{2,i} T_{j,1}^{-1}\end{equation} for any $1\leq j\leq i-1$.
Note that $E_{i, i+1}$ is not the $E_i$ in Definition~\ref{k2222} except the classical limit case.
 %The aim of this subsection is to construct the JM basis elements for $\mathcal B_n$.
%\label{JM1}

\begin{Lemma}\label{ln} If  $n\ge 2$, then  $L_n=T_{n-1}L_{n-1}T_{n-1}+T_{n-1}-q^2z^{-1} E_{n-1, n}$. \end{Lemma}
\begin{proof} The result  follows immediately from \eqref{JM1}.\end{proof}
\begin{Lemma}\label{ln1} Suppose $n>3$ and   $x_m=E_{m-2,m}+E_{m-1,m}$ for $m\ge 3$. Then
  \begin{multicols}{2}
  \item[(1)]
  $E_1 E_{n-1, n}=E_{n-1,n}E_1$, \item[(2)] $T_{m-2} x_m=x_m T_{m-2}$,
 \item[(3)]$T_{n-2} E_{i, n}=E_{i, n} T_{n-2}$ if $1\le i\le n-3$,
 \item[(4)] $T_jE_{n-1, n}=E_{n-1, n}T_j$ if $1\le j\le n-3$.
\end{multicols}
\end{Lemma}
\begin{proof}Thanks to Definition~\ref{qqq}(4)(5) and Lemma~\ref{k2}(3), $$\begin{aligned}   E_1E_{n-1, n}&
=T_{3,n-1}^{-1}T_{n,4} E_1 T_{2}^{-1} T_1^{-1} T_3T_2E_1T_{2,n} T_{n-1,1}^{-1}\\
&=T_{3,n-1}^{-1}T_{n,4} E_1 T_{2}^{-1} T_1^{-1} T_3T_2E_1T_{4,n} T_{n-1,3}^{-1}\\
&=E_{n-1, n} E_1.\\
\end{aligned}$$
This proves    (1).
 One can easily verify (2)-(4) by braid relations in \eqref{kkk1} and Definition~\ref{k2222}(4).\end{proof}

\begin{Lemma}\label{k1}Suppose $n\ge 2$.  For any $x\in \mathcal B_{n-1}$,  $L_nx=xL_n$. \end{Lemma}
\begin{proof} We prove the result by induction on $n$.
 The result is trivial when $n=2$.
 We have \begin{equation}\label{l3} L_3=T_2T_1T_2+T_2-q^2z^{-1} (E_{1,3}+E_{2, 3}).\end{equation}
 It is easy to verify $E_1L_3=L_3E_1=0$.
 It is known that $T_2+T_2T_1T_2$ is the corresponding Jucys-Murphy element of $\mathcal H_3$ (e.g., \cite{Ma}) and hence  commutes with $T_1$. Then
    by  Lemma~\ref{ln1}(2) and \eqref{l3},   $L_3 T_1=T_1L_3$. This proves  the result when $n=3$. Suppose $n>3$. It is enough to prove that $L_nx=xL_n$ for any $x\in \{E_1, T_1, \ldots, T_{n-2}\}$.
By induction assumption on $n-1$, Definition~\ref{qqq} and Lemma~\ref{ln1}(1),
$E_1$ commutes with both $T_{n-1}L_{n-1}T_{n-1}+T_{n-1}$ and $E_{n-1, n}$.
By Lemma~\ref{ln},
$E_1L_n=L_nE_1$.   By induction assumption on $n-1$ and  Lemma~\ref{ln1}(4), $T_1, T_2, \ldots, T_{n-3} $ commute with both
 $T_{n-1}L_{n-1}T_{n-1} +T_{n-1}$ and $E_{n-1, n}$. Using Lemma~\ref{ln} again yields $L_n x=xL_n$ for any
 $x\in \{T_1, T_2, \ldots, T_{n-3}\}$.
It is known that  $T_{n-2}\sum_{i=1}^{n-1} (i, n)=\sum_{i=1}^{n-1} (i, n) T_{n-2}~$ 
since $\sum_{i=1}^{n-1} (i, n)$ is the corresponding Jucys-Murphy  element of $\mathcal H_{n}$.   
 Then by Lemma~\ref{ln1}(2)(3), $T_{n-2}L_n=L_nT_{n-2}$.
 \end{proof}
 \begin{Lemma}\label{JMA}For any $1\le i\le j\le n$, $L_iL_j=L_jL_i$.  \end{Lemma}
 \begin{proof} Obviously, $L_i\in \mathcal B_{j-1}$ if $i<j$.
 Now, the result follows from  Lemma~\ref{k1}.
  \end{proof}

 Thanks to Lemma~\ref{JMA}, the subalgebra $L$ of $\mathcal B_n$
 generated by $\{L_1, L_2, \ldots, L_n\}$ is a commutative subalgebra.
 By Remark~\ref{limit},  the Brauer algebra $B_n(a)$  can be considered as the classical limit of
   $\mathcal B_n(q, q^a)$.
 In this case, $\{L_i\mid 1\le i\le  m\}$
 is the set of corresponding  Jucys-Murphy  elements in \cite{Eng}  and   $\sum_{i=1}^n L_i$ is  a central element.
 However,  $\sum_{i=1}^n L_i$ may not be  a central element  in general. One can easily find a counterexample in $\mathcal B_4$.
%For example,    $\sum_{i=1}^4 L_i$ is not a central element.

\subsection{A new basis  of a cell module}
Suppose $(f, \lambda)\in \Lambda_n$.
Write $\mu\rightarrow \lambda$ if $\mu\in \mathcal {RA}(\lambda)$.
 An up-down tableau $\mathbf t$  of type $\lambda$ is a sequence of partitions $\mathbf t=(\mathbf t_0, \mathbf t_1,\ldots,\mathbf t_n)$ such that
$\mathbf t_0=\emptyset$, $\mathbf t_i\rightarrow\mathbf t_{i+1}$, $0\leq i\leq n-1$  and  $\mathbf t_n=\lambda$.
Let $\mathscr T_n^{ud}(\lambda)$ be the set of all up-down tableaux of type $\lambda$.
There is a partial order $ \preceq $ on $\mathscr T_n^{ud}(\lambda)$
such that    $\mathbf t\prec \mathbf s $ if $\mathbf t_k\lhd\mathbf s_k$ for some $k$
and $\mathbf t_j=\mathbf s_j$ for all $k<j\leq n$. In this  case, we also write $\mathbf t\overset k \prec \mathbf s $.

\begin{Defn}\label{murr} Suppose $(f, \lambda)\in \Lambda_n$, $\mathbf t\in \mathscr T_n^{ud}(\lambda)$ and $\mathbf t_{n-1}=\mu$. Define $\mathrm m_{\mathbf t}=E^f x_\lambda b_{\mathbf t}$, where  $b_{\mathbf t}=b_{\mathbf t_{n}}$ which can be defined inductively as follows:
$$ b_{\mathbf t_n}=\begin{cases} T_{a_k, n} b_{\mathbf t_{n-1}} &\text{if $\lambda=\mu\cup (k, \mu_k+1)$,}\\
T_{n,2f}^{-1} \sum_{j=b_{k-1}+1}^{b_k} q^{b_k-j} T_{j, 2f-1}^{-1} b_{\mathbf t_{n-1}} &\text{if $\lambda=\mu\setminus (k, \mu_k)$,}\\
\end{cases}$$ and  $a_k=2f+\sum_{j=1}^k \lambda_j$ and  $b_i=2f-1+\sum_{j=1}^i \lambda_j$, for $i\in \{k-1, k\}$.

\end{Defn}
In fact one can easily check that \begin{equation}\label{mpbb}
\mathrm m_{\mathbf t_n}=\begin{cases}  \sum_{j=a_{k-1}+1}^{a_k}T_{j,n}\mathrm m_{\mathbf t_{n-1}} &
\text{if $ \lambda=\mu\cup (k,\mu_k+1)$,}\\
 E_{2f-1}T_{n,2f}^{-1}T_{b_k,2f-1}^{-1}\mathrm m_{\mathbf t_{n-1}} &\text{if
  $\lambda=\mu\setminus(k,\mu_k)$,}\end{cases}\end{equation}
and $\mathrm m_{\mathbf t_n}=y_\mu^\lambda b_{\mathbf t_{n-1}}$ in any case.
\begin{Theorem}\label{murphyb}
Suppose $(f, \lambda)\in \Lambda_n$.
\begin{itemize}
\item[(1)]The cell module $C(f, \lambda)$ has $\Bbbk$-basis  $\{\mathrm m_{\mathbf t}+\mathcal B_{n}^{\rhd \lambda}\mid \mathbf t\in \mathscr T_n^{ud}(\lambda) \}$.
\item[(2)] For any  $1\leq j\leq m$,
 the $\Bbbk$-module $M_j$  spanned by
$\{ \mathrm m_{\mathbf t} + \mathcal B_{n}^{\rhd \lambda} \mid \mathbf t\in \mathscr T_n^{ud}(\lambda),
\mathbf t_{n-1}\unrhd \mu^{(j)}\}$  is a $\mathcal B_{n-1}$-submodule of $C(f, \lambda)$, where $\mu^{(j)}$ is  given  in \eqref{ppp11}.
\item[(3)] For any  $1\leq j\leq m$, $M_j=N_j$  where $N_j$'s are given in Definition~\ref{nkk}.
    \item[(4)] As $\mathcal B_{n-1}$-modules, $M_j/M_{j-1} \cong  C(\ell,  \mu^{(j)})$, where
$\ell=f$ if $1\le j\le a$ and $\ell=f-1$ if $a+1\le j\le m$.
The required isomorphism  sends $ \mathrm m_{\mathbf t}+M_{j-1}$ to
 $ \mathrm m_{\mathbf t\downarrow_{n-1}}+ \mathcal B_{n-1}^{\rhd \mu^{(j)}}$,
 where $\mathbf t\! \downarrow_{n-1}=(\mathbf t_0,\mathbf t_1,\ldots,\mathbf t_{n-1})\in
  \mathscr T_{n-1}^{ud}(\mu^{(j)})$.

\item [(5)]  The $\Bbbk$-algebra
 $\mathcal B_n$ has  cellular basis
$$\{ \mathrm m_{\mathbf s \mathbf t}  \mid  \mathbf s, \mathbf t\in \mathscr T^{ud}_n(\lambda),\
(f, \lambda)\in \Lambda_n\}$$
where $\mathbf m_{\mathbf s \mathbf t}=\sigma(b_{\mathbf s}) E^f x_\lambda b_{\mathbf t}$ and
the required  $\sigma$  is given in Lemma~\ref{anti}.

 \end{itemize}
\end{Theorem}\begin{proof}Thanks to Lemmas~\ref{fiels1},~\ref{fiels2},
the  $\Bbbk$-linear map $\phi: C(\ell, \mu^{(j)})\rightarrow N_j/N_{j-1}$ sending
   $E^\ell x_{\mu^{(j)}} T_{d(\t)}T_d+\mathcal B_{n-1}^{\rhd \mu^{(j)}}$ to
   $y^\lambda_{\mu^{(j)}}  T_{d(\t)}T_d+N_{j-1}$ is the required $\mathcal B_{n-1}$-isomorphism. In particular,  $\phi(E^\ell x_{\mu^{(j)}}+\mathcal B_{n-1}^{\rhd \mu^{(j)}})= y_{\mu^{(j)}}^\lambda+N_{j-1}$.
    By induction assumption on $n-1$,  $C(\ell, \mu^{(j)})$ has basis
   $$\{E^\ell x_{\mu^{(j)}} b_{\mathbf u}+\mathcal B_{n-1}^{\rhd \mu^{(j)}} \mid \mathbf u\in \mathscr T^{ud}_{n-1}(\mu^{(j)})\}.$$
  Note that
 $\phi^{-1} (\mathrm m_{\mathbf u} +\mathcal B_{n-1}^{\rhd \mu^{(j)}})=\mathrm m_{\mathbf t}+N_{j-1}$ where
 $ \mathrm m_{\mathbf t}=y_{\mu^{(j)}}^\lambda b_{\mathbf u}$ and $\mathbf t\!\downarrow_{n-1}=\mathbf u$.
   Using the  isomorphism $\phi$ and induction assumption on $N_{j-1}$,
   we see that  $N_{j}$ has  basis
   $\{\mathrm m_{\mathbf t}+\mathcal B_{n}^{\rhd \lambda}\mid \mathbf t \in \mathscr T_n^{ud}(\lambda), \mathbf t_{n-1}\unrhd \mu^{(j)}\}$. In particular, $M_j=N_j$ where $M_j$ is defined in (2). This proves (1)--(4).
 Note that $\mathcal B_n$ is a cellular algebra. Thanks to Lemma~\ref{k2}(2), $E^f$ is fixed by the anti-involution $\sigma$.
 %The proof of it depends on the invertibility of $\delta$~\cite{N18}. It is the reason why we assume $\delta$ is invertible when we recall  the definition of $\mathcal B_n$.
 It is easy to verify $\sigma(x_\lambda)=x_\lambda$. By (1),
we have the corresponding results for left cell modules. This implies (5) by using standard arguments on cellular algebras (see \cite[Theorem~2.7]{RS3} for Brauer algebras).
\end{proof}

%Later on, we denote $\mathrm m_{\t_n}$ by $\mathrm m_\t$.
The basis in Theorem~\ref{murphyb}(1) (resp., (5)) is
 called the \textsf{Jucys-Murphy basis} of $C(f, \lambda)$ (resp., $\mathcal B_n$).
 %We are going to prove that the Jucys-Murphy elements we defined are  the JM elements of $\mathcal B_{n}$
 %with respect to  the Jucys-Murphy basis
% in the sense of \cite[Definition~2.4]{MA1}.

\begin{Lemma}\label{ttt1} Suppose  $0\le f\le \lfloor \frac n2\rfloor$.
 \begin{itemize}\item[(1)] $E^{f} E_{i, n}\in \mathcal B_n^{f+1}$ if $2f+1\le i\le n-1$.
 \item[(2)] $E^f((2j-1, n)+(2j, n)-q^2z^{-1} E_{2j-1, n}-q^2z^{-1} E_{2j, n})=0$ if $1\le j\le f$. \end{itemize}
\end{Lemma}
\begin{proof} Since $i\ge 2f+1$, we have $$\begin{aligned} E^f E_{i, n}&=E^f T^{-1}_{1, i} T_{2, n} E_1 T_{2, n} T_{i, 1}^{-1}\\
&=T_{2f+1, i}^{-1} E^f T_{1, 2f+1}^{-1} T_{2, 2f+2} E_1 T_{2f+2, n}T_{2, n} T_{i, 1}^{-1}\\
&\equiv 0 \pmod {\mathcal B_n^{f+1}}\end{aligned}$$
where the last equivalence  (1) follows from   Lemma~\ref{k2}(1).
  Obviously, (2) follows from \eqref{zzz123} as follows:
 \begin{equation}\label{zzz123} E^f (2f-1,n)= q E^f (2f,n) T_{2f-1},\ E^f  E_{2f-1, n}=z E^f (2f, n), \ E^f E_{2f, n}=zq^{-1}  E^f (2f, n) T_{2f-1}^{-1}.\end{equation}
   By Lemma~\ref{k2}(6), $$\begin{aligned}  E^f (2f-1,n)& =E^f T_{2f-1} (2f,n) T_{2f-1}\\
  &= q E^f (2f,n) T_{2f-1}\end{aligned}$$ proving the first equality in \eqref{zzz123}.
By Lemma~\ref{k2}(5),
$$\begin{aligned} E^f E_{2f-1,n}
&=E_3E_5\dots E_{2f-1} T_{n,3} E_1 T_2 E_1 T_{2,n}T_{2f,2}^{-1} T_{2f-1,1}^{-1}\\
&=z E^f (2f, n)\\\end{aligned}$$
where the last equality follows from Definition~\ref{qqq}(3), \eqref{kkk1} and  Lemma~\ref{k2}(4).
 This proves the second equality in \eqref{zzz123}. By \eqref{eij}, Lemma~\ref{k2}(4) and the second equality in \eqref{zzz123}, we have
$$\begin{aligned}  E^f E_{2f,n}&=E^f  T_{2f-1}^{-1} E_{2f-1,n} T_{2f-1}^{-1}\\
&=zq^{-1}  E^f (2f, n) T_{2f-1}^{-1},\\
\end{aligned}$$ proving the last equality in \eqref{zzz123}.
\end{proof}
\begin{Prop} \label{ss1} Suppose $(f, \lambda)\in \Lambda_n$ and $\mu^{(k)}=\lambda\setminus p_k$ in \eqref{ppp11}, $1\le k\le a$.
 Then $$y_{\mu^{(k)}}^\lambda L_n\equiv \frac{q^{2c(p_k)}-1}{q-q^{-1}} y_{\mu^{(k)}}^\lambda \pmod {N_{k-1}},$$
 where $c(p_k)=j-i$ if $p_k$ is in $i$-th row and $j$-th column of the Young diagram $[\lambda]$.
 %  \begin{multicols}{2}
 %\item[(1)] $y_{\mu^{(k)}}^\lambda L_n\equiv y_{\mu^{(k)}}^\lambda \sum_{j=2f+1}^{n-1}(j, n) \pmod {\mathcal B_n^{f+1}}$.
%\item[(2)] $y_{\mu^{(k)}}^\lambda L_n\equiv \frac{q^{2c(p_k)}-1}{q-q^{-1}} y_{\mu^{(k)}}^\lambda \pmod {N_{k-1}}$
%\end{multicols}
\end{Prop}
\begin{proof} Thanks to  Lemma~\ref{ttt1}, we immediately have \begin{equation}\label{ddddt} y_{\mu^{(k)}}^\lambda L_n\equiv y_{\mu^{(k)}}^\lambda \sum_{j=2f+1}^{n-1}(j, n) \pmod {\mathcal B_n^{f+1}}.\end{equation}
Now, the result follows from \eqref{ddddt} and well-known
results on the action of Jucys-Murphy elements for  Hecke algebras  in \cite[Theorem~3.22]{Ma}\footnote{ The  Hecke algebra $\mathcal H_n$ is defined via
$(T_i-q)(T_i+1)=0$ in \cite{Ma}.
 Our $q^{-1} T_i$ (resp., $q^2$) is $T_i$ (resp., $q$) in \cite{Ma}.}.
\end{proof}

\begin{Lemma}\label{tt1} Suppose $(f, \lambda)\in \Lambda_n$. We have  $$ E^f x_\lambda T_{n,2f}^{-1} x\equiv 0\pmod{N_a}$$ if
$x\in \{ (i, n)-q\delta_{2f-1, i},  E_{j, n}-z^{-1} (2f-1, j)^{-1}, E_{2f-1, n}-\delta\}$
and   $2f-1\le i\le n-1$, $2f-1< j\le n-1$.
\end{Lemma}
\begin{proof}Thanks to Lemma~\ref{k2},
$$E^f x_\lambda T_{n, 2f}^{-1} (i, n)=\begin{cases} x_\lambda E^f  T_{i+1, n} T_{i, 2f}^{-1} &\text{if $i\neq 2f-1$,}\\ q E^f x_\lambda T_{2f, n} &\text{if $i=2f-1$.}
\end{cases}$$
By  Definition~\ref{deofymulambda} and Lemma~\ref{fiels1}(3),  \begin{equation}\label{cass1} E^f x_\lambda T_{n, 2f}^{-1} (i, n)\in N_a \text{  if $i\neq 2f-1$.}\end{equation}  In the remaining case,
 using the quadratic relation in \eqref{kkk1} to rewrite  $T_{2f, n}$ and using Lemma~\ref{fiels1}(3) again, we have  \begin{equation}\label{invkey}
E^f x_\lambda T_{2f, n}-E^f x_\lambda T_{n, 2f}^{-1}\in {N_a}.\end{equation}
When $x=(i, n)-q\delta_{2f-1, i}$, the result follows from \eqref{cass1}--\eqref{invkey}. Similarly, by  Lemma~\ref{k2}(4)(5), Lemma~\ref{k3}(7), and Definition~\ref{qqq}(1)(2) we have
\begin{equation}\label{cas1} \begin{aligned} E^f x_\lambda T_{n, 2f}^{-1} E_{2f-1, n}
&=E^f x_\lambda T_{1, 2f-1}^{-1} T_{ 2f, 2}  E_1 T_{2, n} T_{2f-1, 1}^{-1}\\ & =
\delta E^f x_\lambda T_{2, 2f} T_{ 2f-1,1}^{-1}   T_{2f, n} \\
&=\delta E^f x_\lambda  T_{2f, n},\\
\end{aligned}\end{equation}
and
%Similarly, by  Definition~\ref{qqq}(2), Lemma~\ref{k2}(4)(5) and Lemma~\ref{k3}(7), we have
\begin{equation}\label{cas2} \begin{aligned} E^f  T_{n, 2f}^{-1} E_{2f, n} &
=E^f  T_{1, 2f+1}^{-1}T_{2f+1,2}E_1 T_{2,n} T_{2f, 1}^{-1}\\
& =E^f T_{2f,1} T_{1, 2f+1}^{-1}E_1 T_{2,n} T_{2f, 1}^{-1}
\\ & =E^f T_{2f-1,1} T_{2, 2f+1}^{-1}E_1 T_{2,n} T_{2f, 1}^{-1}
\\ & =E^f T_{2f}^{-1} T_{2f-1,1} T_{2, 2f}^{-1}E_1 T_{2,n} T_{2f, 1}^{-1}\\ &
=E^f T_{2f}^{-1} E_{2f-1} T_{2f-1,1} T_{2, 2f}^{-1} T_{2,n} T_{2f, 1}^{-1}
\\ & =z^{-1} E^f  T_{2f-1,1}  T_{2,2f}^{-1} T_{2, n}T_{2f, 1}^{-1}\\ &=z^{-1} E^f  T_{2f, n} T_{2f-1}^{-1},
\\ \end{aligned} \end{equation}
and
\begin{equation}\label{cas3} \begin{aligned} E^f T_{n, 2f}^{-1} E_{j, n} &
=E^f T_{j+1, 2f}^{-1} T_{1, j}^{-1} T_{n, j+1}^{-1}  T_{n,2}E_1 T_{2,n} T_{j, 1}^{-1}\\
& =E^f  T_{j+1, 2f}^{-1} T_{1, j}^{-1} T_{j+1, 2} E_1 T_{2,n}   T_{j, 1}^{-1}\\
& =E^f T_{ 2f+1, j+1}^{-1}  T_{ j+1, 2f}^{-1}  T_{1, 2f}^{-1} T_{j+1, 2} E_1 T_{2,n} T_{j, 1}^{-1}\\ &
=E^f T_{2f+1, j+1}^{-1} T_{1, 2f+1}^{-1} T_{2f+1, 2} E_1 T_{2,n} T_{j, 1}^{-1}\\ &
=E^f    T_{1, j+1}^{-1} T_{2f+1, 2}   E_1 T_{2,n} T_{j, 1}^{-1}\\ &
=E^f T_{2f, 1}  T_{1, j+1}^{-1}    E_1 T_{2,n} T_{j, 1}^{-1}\\
&=E^f T_{2f-1, 1}  T_{2, j+1}^{-1}    E_1 T_{2,n} T_{j, 1}^{-1}\\ &
=E^f  T_{2f, j+1}^{-1} T_{2f-1, 1}    T_{2, 2f}^{-1}    E_1 T_{2,n} T_{j, 1}^{-1}
\\
&= E^f  T_{2f, j+1}^{-1} E_{2f-1} T_{2f-1, 1}    T_{2, 2f}^{-1}   T_{2,n} T_{j, 1}^{-1}\\ &
= z^{-1} E^f T_{2f-1, 1} T_{2,2f}^{-1}   T_{2f+1, j+1}^{-1}    T_{2,n} T_{j, 1}^{-1}\\
&=z^{-1} E^f  T_{2,n} T_{2f, j}^{-1} T_{j, 1}^{-1}\\ &
=z^{-1} E^f T_{2, 2f} T_{2f, n}  T_{j, 1}^{-1} T_{2f-1, j-1 }^{-1} \\
&=z^{-1} E^f  T_{2f-1, 1}T_{2f, n} T_{j,1}^{-1} T_{2f-1, j-1}^{-1}\\ &
=z^{-1} E^f T_{2f, n} T_{j, 2f-1}^{-1}  T_{2f-1, j-1}^{-1}
\\
&=z^{-1} E^f T_{2f, n} (2f-1, j)^{-1}.
\end{aligned} \end{equation}
When  $x= E_{j, n}-z^{-1} (2f-1, j)^{-1}$, the result follows from \eqref{invkey} and  \eqref{cas2}-\eqref{cas3}. In the remaining case, the result follows from   \eqref{invkey}--\eqref{cas1}.
\end{proof}
%\label{u0}

\begin{Prop} \label{sss1}Suppose $(f, \lambda)\in \Lambda_n$ and   $\mu^{(k)}=\lambda\cup p_k$, $a< k\le m$. Then   $$y_{\mu^{(k)}}^\lambda L_n\equiv (q-q^{-1})^{-1} ({z^{-2} q^{2-2c(p_k)}-1})
  y_{\mu^{(k)}}^\lambda \pmod {N_{k-1}}.$$
\end{Prop}
\begin{proof}
 Thanks to Lemma~\ref{ttt1}(2) and  Lemma~\ref{tt1}, we have
 \begin{equation}\label{sss111} \begin{aligned} E^f x_\lambda T_{n,2f}^{-1} L_n & =   E^f x_\lambda T_{n,2f}^{-1}
\left(\sum_{2f-1\le i\le n-1} (i, n)-q^2z^{-1} E_{i, n}\right)\\ &\equiv E^f x_\lambda T_{n,2f}^{-1}
 (q -q^2z^{-1} \sum_{2f-1\le i\le n-1}  E_{i,n})
 \\
 &\equiv E^f x_\lambda T_{n,2f}^{-1} (q-q^2z^{-1}\delta -q^2z^{-2} \sum_{2f\le j\le n-1} (2f-1, j)^{-1})\\ & \equiv -q^2 z^{-2}
( E^f x_\lambda T_{n, 2f}^{-1}x-E^f x_\lambda x' T_{n, 2f}^{-1} )+(q-q^2z^{-1}\delta)
 E^f x_{\lambda} T_{n, 2f}^{-1} \pmod {N_a},
 \end{aligned}\end{equation}
 where $$x'=\sum_{2f+1\le i<j\le n} (i, j)^{-1}\ \text{  and  } \ x=\sum_{2f-1\le i<j\le n-1} (i, j)^{-1}.$$
 Write $[k]=\frac{q^k-q^{-k}}{q-q^{-1}}$ for any $k\in \mathbb Z$.
By \cite[Theorem~3.22]{Ma}\footnote{The Jucys-Murphy elements of the Hecke algebra in \cite{Ma} were  constructed via  $(i,j)$.
	When we compute \eqref{u1},
	the arguments in the proof of \cite[Theorem~3.22]{Ma} still work but the coefficients now are obtained from those in \cite{Ma} by replacing  $q$ with $q^{-1}$.},
\begin{equation}\label{u1} E^f
 x_\lambda x'- \sum_{p\in [\lambda]}q^{-c(p)}[c(p)]  E^f x_\lambda \in { \mathcal B_n^{\rhd\lambda}},  \ \
y_{\mu^{(k)} }^\lambda x - \sum_{p\in [\mu^{(k)}]} q^{-c(p)}[c(p)] y_{\mu^{(k)} }^\lambda \in  {N_{k-1}}.\end{equation}
Thanks to \eqref{sss111}-\eqref{u1},
$$\begin{aligned} y_{\mu^{(k)} }^\lambda L_n& \equiv (q-q^2z^{-1} \delta-q^2z^{-2} \frac{1-q^{-2c(p_k)}}{q-q^{-1}}) E^f x_\lambda T_{n,2f}^{-1}h \\ & \equiv
\frac{z^{-2} q^{2-c(p_k)}-1}{q-q^{-1}} y_{\mu^{(k)} }^\lambda  \pmod {N_{k-1}}\end{aligned}$$
 where  $h\in \mathcal B_{n-1}$ such that  $y_{\mu^{(k)} }^\lambda=E^f x_\lambda T_{n,2f}^{-1}h $ (see \eqref{symulambda}).

\end{proof}
\begin{Defn}  Suppose $(f, \lambda)\in \Lambda_n$. For any $\mathbf t\in \mathscr T_n^{ud}(\lambda)$ and $1\le k\le n$, define
$$c_{\mathbf t}(k)=\begin{cases}  \frac{q^{2c(p)}-1}{q-q^{-1}} & \text{if $\mathbf t_k=\mathbf t_{k-1}\cup p$,}\\
\frac{z^{-2} q^{2-c(p)}-1}{q-q^{-1}}& \text{if $\mathbf t_k=\mathbf t_{k-1}\setminus p$.}\\
\end{cases}$$
\end{Defn}

\begin{Theorem}\label{trishactio} Suppose $(f, \lambda)\in \Lambda_n$.
For any $\mathbf t\in\mathscr T_n^{ud}(\lambda)$ and $1\leq k\leq n$,
$$\mathrm m_{\mathbf t} L_k= c_{\mathbf t}(k) \mathrm m_{\mathbf t} +\sum_{\mathbf s\overset k \succ\mathbf t}a_\mathbf s
\mathrm m_{\mathbf s} $$ for some scalars $a_\mathbf s\in\Bbbk$.
\end{Theorem}

\begin{proof}We have $\mathrm m_{\mathbf t} L_1=0$. Suppose $k\ge 2$.
By Theorem~\ref{murphyb}(4) and induction assumption for $n-1$,
we have the required formulae on $\mathrm m_{\mathbf t} L_k$, $1\leq k\leq n-1$.
When $k=n$, the required formula  follows from
Propositions~\ref{ss1},~\ref{sss1} and Theorem~\ref{murphyb}(3)(4).
\end{proof}

  By Theorem~\ref{trishactio},
  $\{L_1, L_2, \ldots, L_n\}$ are JM elements with respect to the
  Jucys-Murphy basis of any cell module $C(f, \lambda)$ in the sense of
   \cite[Definition~2.4]{MA1}.
   %Further, the basis in Theorem~\ref{trishactio}(5) is the Jucys-Murphy basis of
   %$\mathcal B_n$.

Now we assume that $\mathcal B_n$ is defined over $F$. It follows from \cite[Definition~2.3]{GL} that there is a canonical invariant form, say $\psi$ on each cell module $C(f, \lambda)$. Let $$D(f, \lambda)=C(f, \lambda)/\text{rad} C(f, \lambda),$$ where $\text{rad} C(f, \lambda)$ is the radical of
$C(f, \lambda)$ with respect to $\psi$. Then all non-zero $D(f, \lambda)$ consist of all pair-wise non-isomorphic simple $\mathcal B_n$-modules~\cite[Theorem~3.4]{GL}. It follows from \cite[Theorem~4.1]{N21} that $D(f, \lambda)\neq 0$ if and only if $\lambda\in \Lambda_e^+(n-2f)$.

\begin{Cor}\label{corw} Suppose $(f, \lambda)\in \Lambda_n$, $\mu \in \Lambda_e^+(n)$  and $[C(f, \lambda): D(0, \mu)]\neq 0$.
Then  there are $(\mathbf s, \mathbf t)\in
\mathscr T_n^{ud}(\mu)\times \mathscr T_n^{ud}(\lambda)$ such that $c_{\mathbf s}(k)=c_{\mathbf t}(k)$ for all $1\le k\le n$.
\end{Cor}
\begin{proof}
Let $L$ be the  abelian subalgebra of $\mathcal B_n$ generated by $L_1, L_2, \ldots, L_n$.
 Restricting both $C(f, \lambda)$ and $ D(0, \mu)$ to $L$ and using
   Theorem~\ref{trishactio} yields the result as required.  \end{proof}

\section{Restriction and Induction functors} In this section, we consider the $q$-Brauer algebra   over a field $F$.
 Recall the algebra homomorphism
$\phi: \mathcal B_{n-2}\rightarrow \tilde E_1 \mathcal B_{n} \tilde E_1$ in Theorem~\ref{iso111}. Identifying $\mathcal B_{n-2}$ with $ \tilde E_1 \mathcal B_{n} \tilde E_1$ via the algebra isomorphism $\phi$, we have
two functors $\mathcal F: \mathcal B_{n}\text{-mod}\rightarrow \mathcal B_{n-2}\text{-mod}$ and
 $\mathcal G: \mathcal B_{n-2}\text{-mod}\rightarrow \mathcal B_{n}\text{-mod}$ such that
 \begin{equation}\label{fg} \mathcal F(M)=M\tilde E_1, \text{ and }\mathcal G(N)= N\otimes_{\tilde E_1 \mathcal B_{n} \tilde E_1}\tilde E_1 \mathcal B_{n} \end{equation}
for any $\mathcal B_n$-module $M$ and any $\mathcal B_{n-2}$-module $N$.
\begin{Prop}\label{fg1} Suppose $(f, \lambda)\in \Lambda_n$ and $(\ell, \mu)\in \Lambda_{n+2}$.
\begin{itemize}\item [(1)] $\mathcal F\mathcal G=1$,
\item[(2)] $\mathcal F(C(f, \lambda))\cong C(f-1, \lambda)$ where $C(-1, \lambda)$ is defined to be $0$,
\item[(3)] $\mathcal G(C(f, \lambda))\cong C(f+1, \lambda)$,
\item [(4)] $\text{\emph{Hom}}_{\mathcal B_{n+2}} (\mathcal G(C(f, \lambda)), C(\ell, \mu))\cong \text{\emph{Hom}}_{\mathcal B_n} (C(f, \lambda), \mathcal F (C(\ell, \mu)))$.
\end{itemize}
\end{Prop}
\begin{proof} Obviously, (1) follows from Theorem~\ref{iso111}. Recall $\mathcal D_{f, n}$ in Definition~\ref{dlf}(1). For any $d\in \mathcal D_{f, n-2}$, let $\tilde d$ be obtained from $d$ by replacing each factor $s_i$ of $d$ by $s_{i+2}$.
Let $\mathscr D_{f, n}$ be the set of all such $\tilde d$. Then
 $\mathscr D_{f, n}\subset \mathcal D_{f+1, n}$.
 By Theorems~\ref{sstan},~\ref{iso111},
$\tilde E_1 \mathcal B_n\tilde E_1$ has cellular basis
\begin{equation} \label{basis321} \{ qz^{-1} T_1^{-1}T_2\sigma(T_d) E^{f+1} x_{\s\t} T_e\mid d, e\in \mathscr D_{f, n}, \s, \t\in \Std(\lambda), (f, \lambda)\in \Lambda_{n-2}\}\end{equation}
where  $x_\lambda$ corresponds
to the index representation of $\mathcal H_{2f+3, n}$. If we denote by $\tilde C(f-1, \lambda)$ the corresponding cell module of
$\tilde E_1 \mathcal B_{n} \tilde E_1$ with respect to $(f-1, \lambda)$, then
\begin{equation}\label{cellsio}   C(f-1, \lambda)\cong \tilde C(f-1, \lambda)\end{equation}  as $\mathcal B_{n-2}$-modules and the required isomorphism sends any basis element
$E^{f-1} x_\s T_d+\mathcal B_{n-2}^{\rhd \lambda}$ to $qz^{-1} T_1^{-1} T_2 E^{f} \phi(x_\s) \phi (T_d)+{\tilde B}^{\rhd \lambda}$
where ${\tilde B}^{\rhd \lambda}$ is the free $\Bbbk$-module spanned by  $$\bigcup_{(l, \mu)\in \Lambda_n} \{ qz^{-1} T_1^{-1}T_2\sigma(T_{e_1}) E^{l} x_{\s\t} T_{e_2}\mid \s, \t\in \Std(\mu),  e_1, e_2\in \mathscr D_{l-1, n},  \mu\rhd \lambda\}.$$ Now, we compute $\mathcal F(C(f, \lambda))$.
Suppose $d\in \mathcal D_{f, n} $. Thanks to \eqref{aast},
 $E^{f} T_d\tilde E_1$ can be written as a linear combination of elements
$ E^{f} T_e  T_w$ with $w\in \mathfrak S_{3, n}$ and $e=s_{2f, i_{f}} s_{2f-1, j_{f}}\cdots s_{4, i_2} s_{3, j_2}\in \mathscr D_{f-1, n}$.
This shows  $\mathcal F(C(f, \lambda))$ has basis
 $$\{E^{f} x_\t T_d +\mathcal B_n^{\rhd \lambda} \mid \t\in \Std(\lambda), d\in  \mathscr D_{f-1, n}\}. $$
So $\mathcal F_n(C(f, \lambda))\cong \tilde C(f-1, \lambda)$ and the   required
right $\tilde E_1 \mathcal B_n\tilde E_1$-isomorphism is
the linear isomorphism sending $ E^{f} x_{\t} T_{d}+{\mathcal B_n^{\rhd \lambda}}$ to $qz^{-1} T_1^{-1}T_2 E^{f} x_\t T_d+{\tilde B}^{\rhd \lambda} $. Now (2) follows from \eqref{cellsio}.

By definition, $\mathcal G(C(f, \lambda))=\overline {E^fx_\lambda} \mathcal B_{n}
\otimes_{\tilde E_1 \mathcal B_{n+2} \tilde E_1} \tilde E_1 \mathcal B_{n+2}$ where
$\overline {E^fx_\lambda}=E^fx_\lambda +\mathcal B_{n}^{\rhd \lambda} $. There is a right
$\mathcal B_{n+2}$-homomorphism $\psi: \mathcal G(C(f, \lambda))\rightarrow C(f+1, \lambda)$
sending $\overline {E^fx_\lambda}\otimes \tilde E_1 h$ to $\overline {E^{f+1}x_\lambda}  h$ for any $h\in \mathcal B_{n+2}$. Since
$C(f+1, \lambda)$ is a cyclic module generated by $\overline {E^{f+1}x_\lambda}$, $\psi$ is surjective.
By Lemma~\ref{efb},   any element in  $\mathcal G(C(f, \lambda))$
can be written as a linear combination of elements $\overline{ E^f   x_\s}
\otimes \tilde E_1 T_d$, where $\s\in \Std(\lambda)$ and  $ d\in \mathcal D_{f+1, n+2}$.  So the number of all such elements is equal to  the rank of $ C(f+1, \lambda)$. Since we are considering  the $q$-Brauer algebra over a field $F$,
$\dim_F \mathcal G(C(f, \lambda))\le \dim_F C(f+1, \lambda)$, forcing $\psi$ to be an isomorphism.
Finally, (4) follows from the adjoint associativity  of  tensor,  hom functors and the well-known fact that $\mathcal F(C(\ell,\mu))\cong \Hom_{\mathcal B_n}(\tilde E_1 \mathcal B_n, C(\ell,\mu)) $.
\end{proof}

\section{The radical of $C(1, \mu)$}
In this section we consider the $q$-Brauer algebra over the field $F$
containing invertible  $q$ such that $e>n$ where $e$ is the quantum characteristic of $q^2$.
We are going to  describe explicitly the  radical of  cell modules  $C(1, \mu)$ for all $(1, \mu)\in \Lambda_n$.  This result can be considered as the counterpart of
\cite[Theorem~3.4]{DWH} for the Brauer algebra.  From Lemma~\ref{ne1} to the end of Remark~\ref{gramlimit}
we assume that the $q$-Brauer algebra is defined over $\Bbbk$.
%Considering  the classical limit of $\mathcal B_{n}(q, q^b)$, we obtain the corresponding results for the Brauer algebra $B(b)$ over $\mathbb Z$.
The following result is motivated by Doran-Wales-Hanlon's work on Brauer algebra in \cite{DWH}.

\begin{Lemma}\label{bimod1}Let $V$ be the free $\Bbbk$-module with basis $\{  E_1 T_w T_d+\mathcal B_n^2\mid w\in \mathfrak S_{3, n}, d\in \mathcal D_{1, n}\}$.
Then $V$ is an $(\mathcal H_{3,n}, \mathcal B_n)$-bimodule with the natural actions induced by the multiplication in $\mathcal B_n$.\end{Lemma}
\begin{proof} The result follows from Theorem~\ref{sstan} and Lemma~\ref{efb}.\end{proof}

In the remaining part of this section, we write  $a=q-q^{-1}$.
\begin{Lemma}\label{ne1} Suppose $s_{2, j_1} s_{1, i_1}, s_{2, j_2} s_{1, i_2}\in \mathcal D_{1, n}$.
   There are  $h\in \sum_{x\in \mathfrak S_{3, n} }  q^{\mathbb Z}\mathbb Z[a] T_x$, $w\in \mathfrak S_{3, n}$  and
 $b\in \{0, 1, 2\}$  such that
\begin{equation}\label{m0}E_1T_{2,j_1}T_{1, i_1} T_{i_2, 1} T_{j_2, 2} E_1\equiv \delta_{(i_1,j_1),(i_2, j_2)}
 \delta E_1 +  (1-  \delta_{(i_1,j_1),(i_2,j_2)}) z q^b T_{w} E_1 + az h E_1\pmod{\mathcal B^2_n}.
 \end{equation}
\end{Lemma}
\begin{proof} Suppose $k<l$.  Thanks to the 	quadratic relation  in \eqref{kkk1}, and Definition~\ref{qqq},
\begin{equation}\label{zzzzz} T_{k, l} T_{l,k}=1+a\sum_{k+1\le i\le l} T_{i, k+1} T_k T_{k+1, i}, \text{ and } E_1 T_2T_1T_3T_2E_1\equiv az E_1+azqE_1T_3 \pmod {\mathcal B_n^2}.
\end{equation}   If  $l\le i_1$, then  \begin{equation}\label{sssss1} E_{1} T_{2,j_1} T_{l, 2} T_{1, l} T_{j_1,2} E_1  =
T_{l+1,3} E_1 T_{2,j_1}T_1T_{j_1,2} E_1 T_{3,l+1}.\end{equation}  Let $x=E_1 T_{2,j_1}T_1T_{j_1,2}E_1$. By braid relations in \eqref{kkk1}, Definition~\ref{qqq}(3) and \eqref{zzzzz},
\begin{equation}\label{ne3} \begin{aligned}x& =
  E_1 T_2T_1 T_{3,j_1}T_{j_1,3} T_2E_1 \\
&=E_1 T_2T_1 (1+a\sum_{4\le m\le j_1} T_{m,4}T_{j_1,3}) T_2E_1 \\
&= q^2z E_1  +a\sum_{4\le m\le j_1}  T_{m,4} E_1T_2T_1T_3T_2E_1 T_{4,m} \\ & \equiv  q^2z E_1   +a\sum_{4\le m\le j_1}  T_{m,4}( az E_1+azqE_1T_3   ) T_{4,m} \pmod {\mathcal B_n^2}
.
\end{aligned}\end{equation}

Now, we are ready to verify \eqref{m0}.
Suppose $(i_1, j_1)=(i_2, j_2)$. By \eqref{zzzzz} and Definition~\ref{qqq}(3)
$$\begin{aligned} E_1 T_{2, j_1} T_{1, i_1} T_{i_1, 1} T_{j_1, 2} E_1& =
E_{1} T_{2, j_1} T_{j_1,2} E_1+ a E_1T_{2,j_1}\sum_{2\le l\le i_1} T_{l,2} T_1 T_{2, l}  T_{j_1,2}E_1\\
&=\delta E_1+a\sum_{3\le k\le j_1}  T_{k, 3} E_1T_2E_1T_{3,k},+a E_1T_{2,j_1}\sum_{2\le l\le i_1} T_{l,2} T_1 T_{2, l}  T_{j_1,2}E_1\\
&=\delta E_1+azE_1\sum_{3\le k\le j_1}  T_{k, 3} T_{3,k},+a E_1T_{2,j_1}\sum_{2\le l\le i_1} T_{l,2} T_1 T_{2, l}  T_{j_1,2}E_1.\\
\end{aligned} $$
By \eqref{sssss1}-\eqref{ne3}, $E_1 T_{2, j_1} T_{1, i_1} T_{i_1, 1} T_{j_1, 2} E_1$  can be written as the required form
in \eqref{m0}.

Suppose $(i_1,j_1)\neq (i_2,j_2)$. Applying anti-involution $\sigma$ on \eqref{m0}, we see that it is enough to prove \eqref{m0} in two cases:(1) $i_1<i_2<j_2$,
 (2) $i_1=i_2$ and $j_1<j_2$ .
 In the first case, we have
    \begin{equation}\label{n6}
 E_1 T_{2,j_1} T_{1, i_1} T_{i_2, 1} T_{j_2,2} E_1=y T_{3, i_1+2}
\end{equation}
where $y= E_1T_{2, j_1} T_{i_2, 1} T_{j_2, 2} E_1 $. We have
 \begin{equation}\label{y1} y  =\begin{cases} T_{i_2,3}T_{j_2, 4} E_1 T_2T_1T_3T_2E_1T_{4, j_1+2} & \text{if $j_1<i_2$,}\\
 qzE_1T_{j_2, 3}+a  \sum_{3\le m\le j_1} T_{m,3}T_{j_2, 4} E_1T_2T_1T_3T_2E_1T_{4, m+1} &\text{if $j_1=i_2$,}\\
 T_{i_2+1, 3}  E_1 T_{2,j_1} T_1  T_{j_2,2} E_1 &\text{if $j_1>i_2$.}\\
 \end{cases} \end{equation}
 Applying anti-involution $\sigma$ on
 $E_1 T_{2,j_1} T_1  T_{j_2,2} E_1$, we see that  it is enough to  assume $j_1\leq j_2$ when we assume $j_1>i_2$.
  If $j_1<j_2$, then
 \begin{equation}\label{y00} E_1 T_{2,j_1} T_1  T_{j_2,2} E_1=T_{j_2,4} E_1T_{2}T_1T_3T_2E_1 T_{4, j_1+1}.\end{equation}
Under the assumption (1) except $j_1=j_2>i_2$, \eqref{m0}   follows immediately from \eqref{n6}--\eqref{y1} if
 we use \eqref{ne1} to rewrite $E_1T_2T_1T_3T_2E_1$. In the remaining case,
 since $ E_1 T_{2,j_1} T_1  T_{j_2,2} E_1$ has already been computed in \eqref{ne3} and
 \eqref{m0} follows immediately.
  In the second case,
 $$\begin{aligned} E_1 T_{2,j_1} T_{1, i_1} T_{i_1, 1} T_{j_2,1} E_1& =E_1T_{2, j_1} T_{j_2,2} E_1+ a\sum_{2\le m\le i_1}   E_1T_{2,j_1} T_{m, 2} T_{1,m} T_{j_2,2}E_1\\
  & =zE_1T_{j_2,3} T_{3,j_1+1} +a\sum_{2\le m\le i_1}T_{m+1,3} E_1T_2T_1  T_{j_2,2} E_1T_{4,j_1+1} T_{3, m+1}\\ &
 = zE_1T_{j_2,3} T_{3,j_1+1} +a\sum_{2\le m\le i_1} T_{m+1,3} T_{j_2,4}  E_1T_2T_1  T_3T_2 E_1T_{4,j_1+1} T_{3, m+1}.
  \end{aligned}
$$
 Rewriting   $E_1T_2T_1  T_3T_2 E_1$ via  \eqref{ne1} again, we have \eqref{m0} immediately.
\end{proof}

\begin{Defn}\label{bili1}  Let $\phi: V\times V \rightarrow \Bbbk$ be the bilinear form such that for any $w_1, w_2\in \mathfrak S_{3, n}$ and $d_1,d_2\in \mathcal D_{1, n}$,
$$\phi( E_1 T_{w_1} T_{d_1}+\mathcal B_n^2,  E_1T_{w_2} T_{d_2}+\mathcal B_n^2)=\tau(h)$$ where  $h\in \mathcal H_{3, n}$ such that
$E_1 T_{w_1} T_{d_1} \sigma(T_{d_2}) \sigma (T_{w_2}) E_1\equiv E_1h \pmod
{\mathcal B_n^2}$ and $\tau: \mathcal H_{3,n}\rightarrow \mathcal Z $ is the trace function in \eqref{tau}.
\end{Defn}
Thanks to Theorem~\ref{sstan} and Corollary~\ref{e1be1}(1), the element $h$ in Definition~\ref{bili1} is unique  and the bilinear form $\phi$ is well-defined.
The following result follows immediately from  Definition~\ref{bili1}.
 \begin{Lemma}\label{phi}  Let $\phi: V\times V\rightarrow \Bbbk$ be the bilinear form
  in Definition~\ref{bili1}. Then $\phi$ is a symmetric invariant form in the sense that
 $\phi(x a, y)=\phi(x, y\sigma(a))$ for all $x, y\in V $ and $a\in \mathcal B_n$.
\end{Lemma}

In order to simplify notation, define \begin{equation}\label{vwd} v_{w, d_{j, i}}=E_1 T_w T_{d_{j, i}}+\mathcal B_n^2, \end{equation} where   $d_{j, i}=s_{2,j}s_{1, i}\in \mathcal D_{1, n}$ and $w\in \mathfrak S_{3, n}$.
Then $\{v_{w, d_{j, i}}\mid 1\le i<j\le n, w\in \mathfrak S_{{3, n}}\}$ is a basis of $V$.

\begin{Lemma}\label{sys2} For any $1\le k\le n-1$, $T_k$ acts on the right of $V$ as a symmetric matrix. \end{Lemma}
\begin{proof}There are some $ a_{(w',d_{j', i'}),(w,d_{j, i})}\in \Bbbk$ and
$(w',d_{j', i'})\in \mathfrak S_{3, n}\times  \mathcal D_{1, n}$ such that
\begin{equation} \label{sys1} v_{w, d_{j, i}} T_k=\sum_{(w',d_{j', i'})} a_{(w',d_{j', i'}),(w,d_{j, i})} v_{w', d_{j',i'}}.\end{equation}  Let $A_k=(a_{(w',d_{j', i'}),(w,d_{j, i})})$ be the $m\times m$-matrix where
$m$ is the rank of $V$. Then $A_k$ is the matrix with respect to $T_k$.
 The LHS of \eqref{sys1} can be computed explicitly by Definition~\ref{qqq}. There are five cases:  (1) $i=k+1$, (2) $i=k$, (3) $i>k+1$, (4) $k>j$ (5) $i<k\le j$.  %We compute (1) and (2) as examples and leave others to the readers.

Thanks to the  braid relations in \eqref{kkk1} and Definition~\ref{qqq}(2), we have
   $$v_{w, d_{j, i}} T_k=\begin{cases}
(q-q^{-1}) v_{w, d_{j, k+1}}+v_{w, d_{j, k}} & \text{ if $i=k+1$,}\\
v_{w, d_{j, k+1}} &\text{ if $i=k, j>k+1$,}\\
qv_{w, d_{j, k}}&\text{ if $i=k, j=k+1$.}\\
 \end{cases} $$
 When $i>k+1$, $$v_{w, d_{j, i}} T_k=\begin{cases} v_{ws_{k+2},d_{j,i }}  &\text{if $\ell(ws_{k+2})>\ell(w)$,}\\
 v_{ws_{k+2},d_{j,i }} +(q-q^{-1})v_{w, d_{j, i}}   &\text{if $\ell(ws_{k+2})<\ell(w)$.}
 \\
 \end{cases}$$
 If $k>j$, then $$v_{w, d_{j, i}} T_k=\begin{cases}v_{ws_{k}, d_{j,i} } &\text{if $\ell(ws_{k})>\ell(w)$,}\\
 v_{ws_{k}, d_{j, i}}+(q-q^{-1}) v_{w, d_{j, i}}   &\text{if $\ell(ws_{k})<\ell(w)$.}
 \\
 \end{cases}$$
 Finally, assume $i<k\le j$. If $k\ge j-1$, then
 $$v_{w, d_{j, i}} T_k=\begin{cases} v_{w, d_{j+1, i}  } &\text{if $k=j$,}\\
 (q-q^{-1})v_{w, d_{j, i} } + v_{w, d_{j-1, i} }   &\text{if $k=j-1$.}
 \\
 \end{cases}$$
 In the remaining case, $k<j-1$. We have
 $$v_{w, d_{j, i}} T_k=\begin{cases} v_{ws_{k+1}, d_{j, i}  } &\text{if $\ell(ws_{k+1})>\ell(w)$,}\\
 (q-q^{-1})v_{w, d_{j, i} } + v_{ws_{k+1}, d_{j, i} }   &\text{if $\ell(ws_{k+1})<\ell(w)$.}
 \\
 \end{cases}$$
 In summary, it follows from explicit computation on  $v_{w, d_{j, i}} T_k$ that  each
diagonal entry of $A_k$ is one of scalar in $\{q, q-q^{-1}, 1\}$ and each  off diagonal entry  of $A_k$ is either $0$ or $ 1$. Further, for any  $(w',d_{j', i'})\neq (w,d_{j, i})$, $a_{(w',d_{j', i'}),(w,d_{j, i})}=1$  if and only if $a_{(w,d_{j, i}), (w',d_{j', i'})}=1$. In particular,
  $A_k$ is a symmetric matrix, proving the result.
\end{proof}

\begin{Lemma}\label{bim1}  Let $G$ be the Gram matrix with respect to  the basis  $\{v_{w, d_{j, i}}\mid 1\le i<j\le n, w\in \mathfrak S_{{3, n}}\}$ of $V$  and  the symmetric invariant form $\phi$ in Definition~\ref{bili1}.
 If $\psi: V\rightarrow V$ is the linear map given by $G$ with respect to the basis above,  then  $\psi $ is a  $(\mathcal H_{3,n}, \mathcal H_n)$-bimodule homomorphism.
\end{Lemma}
\begin{proof}Thanks to Lemma~\ref{phi}, $GA_k$ is a symmetric matrix. By  Lemma~\ref{sys2},
$G A_k=A_{k}^t G^t=A_{k} G$, which is equivalent to saying that $\psi$ is a right $\mathcal H_n$-homomorphism.  Suppose $3\le i\le n$ and $(w, d), (w', d')\in \mathfrak S_{3, n}\times \mathcal D_{1, n}$.
Thanks to Definition~\ref{bili1} there is an $h\in \mathcal H_{3, n}$ such that
 $$\phi(T_i E_1 T_w T_d, E_1 T_{w'} T_{d'})=\tau(T_i h) \ \text{ and $\phi( E_1 T_w T_d, T_i E_1 T_{w'} T_{d'})=\tau(h T_i)$}.$$ Since $\tau$ is a trace function,
we have  $\phi(T_i E_1 T_w T_d, E_1 T_{w'} T_{d'})=\phi( E_1 T_w T_d, T_i E_1 T_{w'} T_{d'})$, proving that $\psi$ is a left $\mathcal H_{3,n}$-homomorphism.
Further, $\psi$ is a $(\mathcal H_{3,n}, \mathcal H_{n})$-homomorphism since two actions are induced by the multiplication on $\mathcal B_n$.
\end{proof}

\begin{Prop} \label{key321} For any two   basis elements $v_{w, d}, v_{w',d'}$ of $ V$ in \eqref{vwd},
$$\phi(v_{w, d}, v_{w',d'})=\begin{cases}
                                \delta+ az f(q) & \text{if $(w,d)=(w',d')$,} \\
                                czq^b+azg(q) & \text{otherwise,}
                              \end{cases}
$$
for some $b\in \mathbb Z$, $c\in \{0,1\}$ and $f(q),g(q)\in q^{\mathbb Z}\mathbb Z[a]$.
\end{Prop}
\begin{proof} The result follows from \eqref{tau} and  \eqref{m0}.
 \end{proof}

\begin{rem}\label{gramlimit}
 Suppose $z=q^m$ for some $m\in \mathbb Z$. The classical limit of $\mathcal B_n$ is the  Brauer  algebra $B_n(m)$ over $\mathbb Z$ (see Remark~\ref{limit}).
In this case,  the above  bilinear form  becomes the corresponding bilinear form on $V$ for $B_n(m)$ by setting $q\rightarrow 1$ and
 the corresponding Gram matrix  is $G_1=\lim_{q\rightarrow1}G$.
 More explicitly, the diagonal (resp., off diagonal) entries of the Gram matrix $G_1$ are of form $\delta$ (resp., $ 0$ or $1$) and
 the linear map  $\psi: V\rightarrow V$ induced by $ G_1$ in Lemma~\ref{bim1} becomes an $(\mathfrak S_{3,n},\mathfrak S_n)$-bimodule homomorphism.
\end{rem}

From here to the end of this section, we assume that   $\mathcal B_n$ is defined over the field $F$ which contains  a non-zero $q$ such that $e>n$ where $e$ is the quantum characteristic of $q^2$.

\begin{Defn} \label{admis} Suppose $e>n$ and  $(1, \mu), (0, \lambda)\in \Lambda_n$.
We say that $\lambda$ is $(1, \mu)$-admissible if \begin{itemize}
 \item[(1)] $\mu\subset \lambda$ and
 $\lambda\setminus \mu=\{p_1, p_2\}$ such that $p_1, p_2$ are not at the same column, 
 \item[(2)] $z^2=q^{2-2(c(p_1)+c(p_2))}$.\end{itemize}\end{Defn}

%The following result is the counterpart of \cite[Theorem~3.4]{DWH}.
\begin{Theorem}\label{semi1}
For any   $(1, \mu), (0, \lambda)\in \Lambda_n$, $[C(1, \mu): C(0, \lambda)]\neq 0$
 if and only if  $\lambda$ is $(1, \mu)$-admissible. \end{Theorem}
\begin{proof}
 Since we are assuming  $e>n$, $\mathcal H_n$ is semisimple and any  cell module $S^\lambda$ of $\mathcal H_n$ is irreducible.
 Recall the cell module $C(1,\mu)$ has a basis $\{E_1x_\mu T_{d(\t)}T_v+\mathcal B_n^{\rhd \mu} \mid (v,\t)\in I(\mu)\}$.
 So, it is easy to check that  $C(1, \mu)\cong  \text{Ind}_{\mathcal H_{2}\otimes\mathcal H_{3, n}
 }^{\mathcal H_n}M\boxtimes N$ as $\mathcal H_n$-modules, where $M$ (resp., $N$) is the $\mathcal H_{2}$-module (resp., $\mathcal H_{3, n}$-module) with basis $\{E_1\}$ (resp., $\{x_\mu T_{d(\t)} +\mathcal H_{3,n}^{\rhd \mu}\mid \t\in \Std(\mu)\}$). Since $M\cong S^\alpha$ with $\alpha=(2)$ (see Definition~\ref{qqq}(2)) and $N\cong S^\mu$, we have  
 $$ C(1,\mu)\cong \text{Ind}_{\mathcal H_{2}\otimes
 \mathcal H_{3, n}}^{\mathcal H_n} S^\alpha \boxtimes S^{\mu}$$
 as right $\mathcal H_n$-modules.
 Using Littlewood-Richardson rule for semisimple Hecke algebra $\mathcal H_n$,
 \begin{equation}\label{iso1} C(1, \mu)\cong \oplus_{\gamma} S^\gamma \end{equation}
 as $\mathcal H_n$-modules,
where $\gamma\in \Lambda^+(n)$   obtained from  $\mu$ by  adding two boxes which are  not at the same column. Moreover,
 the multiplicity of $S^\gamma$ is $1$. So $[C(1, \mu):C(0, \lambda)]\neq 0$ if and only if $[C(1, \mu):C(0, \lambda)]=1$.
  Further, by Corollary~\ref{corw}, $\lambda\setminus \mu=\{p_1, p_2\}$ such that
  \begin{equation}\label{mmm1}  \frac {q^{2c(p_1)}-1}{q-q^{-1}}=\frac{z^{-2} q^{2-c(p_2)}-1}{q-q^{-1}}.\end{equation}
Therefore
  $\lambda$ is $(1, \mu)$-admissible.

We are going to prove $[C(1, \mu): C(0, \lambda)]\neq 0$  if  $\lambda$ is $(1, \mu)$-admissible.
Consider the $(\mathcal H_{3, n}, \mathcal H_n)$-bimodule homomorphism induced by  $\psi: V\rightarrow V$ in Lemma~\ref{bim1} via base change.
 Since we are assuming that $e>n$, both $\mathcal H_n$ and $\mathcal H_{3, n}$ are semisimple. This implies that $V$ is direct sum of cell modules
$C(1, \mu)$ with multiplicity $\dim S^\mu$. Since $\psi$ is a left $\mathcal H_{3, n}$-homomorphism,
the restriction of $\psi $ to $C(1, \mu)$ fixes  $C(1, \mu)$. Let $e_\mu$ be the primitive idempotent  with respect to $S^\mu$.
 Then $S^\mu=e_\mu \mathcal H_n$ and  $\tau(e_\mu)\neq 0$ (see \cite[Lemma~1.6]{Maa}).
 Since $\phi( E_1 e_\mu, E_1 e_\mu )= \delta \tau(e_\mu)\neq 0$,
the restriction of $\psi$ to the cell module $C(1, \mu)$ is non-trivial and hence
\begin{equation}\label{nnn1}  \text{Ker} \psi \cap C(1, \mu) \neq C(1, \mu)\end{equation}
Thanks to \eqref{iso1}, $C(1, \mu)=\bigoplus C(1, \mu)_\gamma$ where $C(1, \mu)_\gamma$ is the right
 $\mathcal H_n$--submodule such that  $$C(1, \mu)_\gamma\cong S^\gamma.$$
  Then
 $C(1, \mu) e_\gamma\mathcal H_n= C(1, \mu)_\gamma$. Since $\psi$ is a right $\mathcal H_{n}$-homomorphism,
  $\psi$ fixes each $C(1, \mu)_\gamma$.
  Pick up a basis of   $C(1, \mu)_\gamma$ for each $\gamma$. The disjoint union of such bases
    forms a
   new basis of $C(1, \mu)$. Note that each $C(1, \mu)_\gamma$ is a right irreducible
  $\mathcal H_n$-module.  With respect to this
 basis, the corresponding matrix with respect to the restriction of $\psi$ to $C(1, \mu)$ is of form
  $\text{diag} (h_\gamma I_\gamma)_\gamma$ where $I_\gamma$ is the
 $\dim S^\gamma\times \dim S^\gamma $ identity matrix and $h_\gamma\in F$.

 Recall that $\{v_{w, d}\mid w\in \mathfrak S_{3, n}, d\in \mathcal D_{1, n}\}$ is a basis of $V$
  and the disjoint union of bases of  $C(1, \mu)_\gamma$ form a new basis of $V$. Moreover, since our computation only involves
Hecke algebra,   each entry in the transition matrix between these two bases of $V$ is of form $f(q)\otimes 1_F$ for some  $f(q)=\frac{g(q)}{h(q)}$  where
  $g(q), h(q)\in \mathbb Z[q]$  such that $ h(q)\otimes 1_F\neq 0$.
  Abusing of notation, we will denote  $a\otimes_\Bbbk 1_F$ by $a$.
 Thanks to Proposition~\ref{key321} each entry in the diagonal (resp., off diagonal) of the  matrix $G$  of $\psi$ with respect to the basis
   is of form $ \delta +zg(q)$ (resp., $zf(q)$) where
 $f(q), g(q)\in F$.
 Let $m=\dim V$.
 So, there exist $f_i(q)\in F$,  $1\leq i\leq m$ with $f_j(q)\neq 0$ for some $j$ such that
 $$G(f_1(q),\ldots,f_m(q))^t=h_{\gamma}(f_1(q),\ldots,f_m(q))^t$$
 for each $\gamma$ in \eqref{iso1}.
 Then $(\delta+zg(q))f_j(q)+zf(q)=h_\gamma f_j(q) $ for some $f(q), g(q)\in F$ and hence
 \begin{equation}\label{skiwjs}
 h_\gamma= \delta+zf_\gamma(q)
 \end{equation}
 for some $f_\gamma(q)\in F$. We explain that there are $z$ and $q$ such that
 $h_\lambda=0$. This can be seen as follows.

 Let  $z=q^{b}$ for some $b\in \mathbb Z$ and consider the classical limit of $\mathcal B_n$. It is the Brauer algebra $B_n(b)$ over $\mathbb Z$.
 The corresponding Brauer algebra $B_n(b)$ over $\mathbb C$ is $B_n(b)\otimes_{\mathbb Z} \mathbb C$.  Unless otherwise stated, we work on Brauer algebra
 $B_n(b)$ over $\mathbb C$ at moment.
 Then
 $$(\lim_{q\rightarrow1}f_1(q),\ldots,\lim_{q\rightarrow1}f_m(q)),$$
 is the corresponding eigenvector of $G_1$ in Remark~\ref{gramlimit} (In fact, it is the corresponding $G_1$ over $\mathbb C$).
Further, $\text{lim}_{q\rightarrow 1}\delta =b$ and
 $\lim_{q\rightarrow1}h_\gamma=b+h $, where $\lim_{q\rightarrow1}f_\gamma(q)=h$.
We may choose $b$ such that $b+h=0$. This implies that $C(1, \mu)_\lambda$ is in the kernel of the
$\psi$  in classical limit case.  Such a kernel is the corresponding  radical of the invariant form $\phi$.
By \eqref{nnn1}, $C(1, \mu)_\lambda$ is a proper submodule of $C(1, \mu)$. If we
 denote $\text{rad} C(1, \mu)$ by the radical with respect to the canonical invariant form on $C(1, \mu)$,  then  $C(1, \mu)_\lambda\subset \text{rad} C(1, \mu)$. Standard arguments on cellular algebras in  \cite{Ruisi}  shows that $\text{rad} C(1, \mu)$ is killed by $E_1$ and so is $ C(1, \mu)_\lambda$. This proves that  $ C(1, \mu)_\lambda\cong C(0, \lambda)$.
By \cite[Theorem~3.2]{DWH},  $b=1-c(p_1)-c(p_2)$ and  $h=-b$. Since $\delta\neq 0$,
$\lim_{q\rightarrow1}f_\lambda(q)=-b\neq 0$.
So  $f_\lambda(q)\neq 0$. Since $\lim_{q\rightarrow1}f_\lambda(q)$ exist, $(q-q^{-1}) f_\lambda(q)\neq 1$.

We may choose $z$ and $q$ such that $z^{-2}=1-(q-q^{-1})f_\lambda(q)\not\in \{0,1\}$. In other words, we can find $z$ and $q$ such that $\delta$ is invertible and
 $h_\lambda=0$. By arguments similar to those for the Brauer algebra above,
   $C(1, \mu)_\lambda\cong C( 0, \lambda)$ and hence $z^2=q^{2-2(c(p_1)+c(p_2))}$ (see  \eqref{mmm1}).
   \end{proof}
\section{Proof of Theorem~\ref{main3}}
In this section, we give a proof of Theorem~\ref{main3}.
For any $(f, \lambda)\in \Lambda_n$, let $G_{f, \lambda}$ be the Gram matrix associated to the canonical invariant form (see \cite[Definition~2.3]{GL}) on $C(f,\lambda)$ with respect to the cellular basis in Theorem~\ref{sstan}.
\begin{Prop}\label{semisim} For $n\ge 2$, let $\mathcal B_n$ be defined   over  $F$  containing  non-zero $q, z, q-q^{-1}, z-z^{-1}$ such that the quantum characteristic $e$ of $q^2$ is strictly bigger than  $n$.   Then
 $\mathcal B_n$ is (split) semisimple if and only if
$\det G_{1, \lambda}\neq 0$ for all $\lambda\in \Lambda^+(k-2) $ and $2\le k\le n$.
\end{Prop}
\begin{proof}Thanks to Theorem~ \ref{sstan}, $\mathcal B_n$ is a cellular algebra over an arbitrary field.
So any field is a split field of $\mathcal B_n$.
By  \cite[Theorem~3.8]{GL}, $\mathcal B_n$ is  (split) semisimple if and only if
$\prod_{(f, \lambda)\in \Lambda_n} \det G_{f, \lambda}\neq 0$.
Since we are assuming  $e>n$, $\mathcal H_n$ is semisimple
and hence $\prod_{\lambda\in \Lambda^+(n)} \det G_{0, \lambda}\neq 0$ . Consequently,  $\mathcal B_n$ is (split)
 semisimple if and only if $\prod_{(f, \lambda)\in \widetilde{\Lambda_n}} \det G_{f, \lambda}\neq 0$ where $\widetilde {\Lambda_n}=\{ (f, \lambda)\in \Lambda_n\mid f>0\}$.

We prove the result by induction on $n$.  Since we are assuming that $\delta$ is invertible,
$\det G_{1, \emptyset}=\delta\neq 0$.   So $\mathcal B_2$ is always semisimple.  In the following we assume $n>2$ and the result is true for $\mathcal B_r$ with $r\leq n -1$.

``$\Leftarrow$":   If the result  were false, there should be an $(f, \lambda)$ with $f\ge 2$ such that $\det G_{f, \lambda}=0$.
In this case, there is a simple module $D(\ell, \mu)\subset \text{rad} C(f, \lambda)$.
In particular,
$(\ell, \mu)\lhd (f, \lambda)$ in the sense that either $\ell< f$ or $\ell=f$ and $\mu\lhd \lambda$. However,
since $ D(\ell, \mu)\hookrightarrow C(f, \lambda)$, there is a non zero homomorphism from  $ C(\ell, \mu)$ to $ C(f, \lambda)$. Acting the exact functor $\mathcal F$ and using
Proposition~\ref{fg1}(2)-(4) yields a non-zero homomorphism   $ C(0, \mu)\rightarrow C(f-\ell, \lambda)$. We claim $\ell<f$. Otherwise, both  $ C(0, \mu)$ and $C(0, \lambda)$ are simple modules of the semisimple Hecke algebra $\mathcal H_{n-2f}$. So  $\lambda=\mu$ and hence $ [C(f, \lambda): D(f, \lambda)]\ge 2$, a contradiction. Hence $f-\ell >0$.

If $\ell>0$, then  $\det G_{f-\ell, \lambda}=0$, a contradiction since $\mathcal B_{n-2\ell}$
is semisimple by induction assumption. Suppose $\ell=0$.
Thanks to  Theorem~\ref{Burnch}, there are two boxes $p_1$ and $p_2$ such that
either $[C(f, \lambda\setminus p_1): C(0, \mu\setminus p_2)]\neq 0$ or $[C(f-1, \lambda\cup p_1): C(0, \mu\setminus p_2)]\neq 0$.
Since we are assuming $f\ge 2$, $f>f-1\ge 1$. Consequently,  either $\det G_{f, \lambda\setminus p_1}\neq 0$ or
$\det G_{f-1, \lambda\cup p_1}\neq 0$. This is a contradiction since $\mathcal B_{n-1}$ is semisimple by induction assumption.

``$\Rightarrow$"   Suppose $\det G_{1, \lambda}= 0$ for some
  $\lambda\in  \Lambda^+(k), 0\le k\le n-2$. Then $k\neq 0$ and
  there is a monomorphism   $C(0, \mu)\hookrightarrow  C(1, \lambda)$ for some partition $|\mu|=k+2$.  Since $\mathcal B_n$ is semisimple, we have $k<n-2$.
  If $k\equiv n \pmod{2 }$, then we apply the functor $\mathcal G$ repeatedly so as to get a non-zero
  homomorphism from $C(f-1, \mu) $ to $C(f, \lambda)$ where
  $f=\frac12 (n-k)$. So  $\det G_{f, \lambda}=0$ and  $\mathcal B_n$ is not semisimple, a contradiction.
  Suppose  $k\equiv n+1 \pmod{2}$. Since   $D(0, \mu)\hookrightarrow  C(1, \lambda)$,
  by Theorem~\ref{semi1}, $\mu$ is $(1, \lambda)$-admissible.
  By Definition~\ref{admis}, $z^2=q^{2(1-c(p_1)-c(p_2))}$ where $\mu=\lambda\cup \{p_1, p_2\}$.
  The arguments in the proof of \cite[Theorem~4.1]{Rui} shows that  there is  a pair of partitions
  $\tilde \lambda$ and $\tilde \mu$ such that $(0, \tilde \mu)$ is $(1, \tilde \lambda)$-admissible
  and    $|\tilde\lambda|=|\lambda|\pm 1$. By Theorem~\ref{semi1}, there is a non-zero homomorphism
  from $C(0, \tilde \mu)$ to $C(1, \tilde\lambda)$.
   Applying the functor $\mathcal G$ repeatedly yields
   $\det G_{f, \tilde \lambda}=0$ where
  $f=\frac 12(n-|\tilde \lambda|)$. So $\mathcal B_n$ is not semisimple, a contradiction.
    \end{proof}

%\begin{Theorem}\label{main2} Suppose $\mathcal B_n$ is defined over the field $F$. Then
% $\mathcal B_n$ is semisimple if and only if
%$o(q^2)>n$ and $z^2\neq q^{2a}$ and $a\in \{i\in \mathbb Z\mid 4-2n\le i\le n-2\}\setminus \{i\in \mathbb Z\mid
%4-2n<i\le 3-n, 2\nmid i\} $.
%\end{Theorem}

\noindent\textbf {Proof of Theorem~\ref{main3}}: Thanks to \eqref{quot1},   $\mathcal H_n \cong \mathcal B_n/\langle E_1\rangle$, where $\langle E_1\rangle $ is the
two-sided ideal generated by $E_1$. Therefore, $\mathcal B_n$ is not semisimple if $\mathcal H_n$ is not semisimple. The later is
equivalent to the condition $e\le n$. In summary, we can assume $e>n$ when we discuss semisimplicity of
$\mathcal B_n$ over an arbitrary field. In this case, Proposition~\ref{semisim} gives a necessary and  sufficient
condition for $\mathcal B_n$ being semisimple over an arbitrary field. Thanks to Definition~\ref{admis} and
Theorem~\ref{semi1}, $\mathcal B_n$ is semisimple if and only if
$z^2\neq q^{2a}$ where $$a\in \bigcup_{k=2}^n \{ 1- \sum_{p\in \lambda\setminus \mu} c(p)
\mid \lambda\in \Lambda^+(k),  \mu\subset\lambda,\lambda\setminus \mu =\{p_1, p_2\} \text{ and $p_1, p_2$ are not at the same column} \}$$
By  \cite[Lemma~2.4]{Ruisi}, this is  equivalent to
saying that  $a$ is given in \eqref{paracon}.\qed

 %\{i\in \mathbb Z\mid 4-2n\le i\le n-2\}\setminus \{i\in \mathbb Z\mid
%4-2n<i\le 3-n, 2\nmid i\}$, proving the result.\qed

\end{document}